\documentclass[12pt,a4paper]{amsart}

\usepackage[left=3.3cm,right=3.3cm,top=3cm,bottom=3cm]{geometry}
\usepackage[utf8]{inputenc}
\usepackage{amsmath}
\usepackage{amsfonts}
\usepackage{amssymb}
\usepackage{amsthm}
\usepackage[all]{xy}
\usepackage{color}
\usepackage{hyperref}
\usepackage{xypic}
\usepackage{mathrsfs}
\usepackage{centernot}
\usepackage{tikz-cd}

\theoremstyle{plain}
\newtheorem{theorem}{Theorem}

\newtheorem{lemma}[theorem]{Lemma}

\DeclareMathOperator{\GL}{GL}

\DeclareMathOperator{\Hom}{Hom}

\DeclareMathOperator{\im}{im}

\DeclareMathOperator{\SL}{SL}

\DeclareMathOperator{\Stab}{Stab}

\newcommand{\del}{\partial}
\newcommand{\delbar}{\overline{\partial}}

\newcommand{\IN}{\mathbb{N}}

\newcommand{\IZ}{\mathbb{Z}}

\newcommand{\IZhat}{\widehat{\IZ}}

\newcommand{\IQ}{\mathbb{Q}}

\newcommand{\IQbar}{\overline{\mathbb{Q}}}
\newcommand{\IR}{\mathbb{R}}
\newcommand{\IC}{\mathbb{C}}

\newcommand{\Gm}{\mathbb{G}_m}
\newcommand{\IP}{\mathbb{P}}

\newcommand{\dprime}{{\prime\prime}}

\newcommand{\pr}{\mathrm{pr}}

\newcommand{\caO}{\mathcal{O}}

\newcommand{\id}{\mathrm{id}}

\newcommand{\hhat}{\widehat{h}}

\usepackage{bm}

\makeatletter
\newcommand\footnoteref[1]{\protected@xdef\@thefnmark{\ref{#1}}\@footnotemark}
\makeatother

\usepackage{BOONDOX-cal}
\usepackage[new]{old-arrows}

\begin{document}
	
\author{Lars K\"uhne}\thanks{The author was supported by an Ambizione Grant of the Swiss National Science Foundation during the early stages of this project. He also received funding from the European Union Horizon 2020 research and innovation programme under the Marie Sklodowska-Curie grant agreement No. 101027237.}
\email{lk@math.ku.dk}
\address{Institut for Matematiske Fag \\
Universitetsparken 5 \\
2100 København Ø \\
Denmark}


\title[The Relative Bogomolov Conjecture]{The Relative Bogomolov Conjecture for Fibered Products of Elliptic Curves}

\begin{abstract}

We deduce an analogue of the Bogomolov conjecture for non-degenerate subvarieties in fibered products of families of elliptic curves from the author's recent theorem on equidistribution in families of abelian varieties. This generalizes results of DeMarco and Mavraki and improves certain results of Manin-Mumford type proven by Masser and Zannier to results of Bogomolov type, yielding the first results of this type for subvarieties of relative dimension $>1$ in families of abelian varieties with trivial trace.

\end{abstract}

\maketitle

In a previous article \cite{Kuehnea}, the author has established an analogue of the equidistribution conjecture for degenerate subvarieties in families of abelian varieties and deduced uniform versions of the Manin-Mumford and the Bogomolov conjecture for algebraic curves embedded in their Jacobian. In this article, we discuss another application of the same equidistribution result \cite[Theorem 1]{Kuehnea}, which has been the original motivation for the author's work on equidistribution. It should also be remarked that since the preprint \cite{Kuehnea} appeared, more general equidistribution results have been obtained by Gauthier \cite{Gauthier2021} as well as Yuan and Zhang \cite{Yuan2021b}.

Pink has suggested a generalization \cite[Conjecture 6.2]{Pink2005a} of the Manin-Mumford conjecture for families of abelian varieties. The following conjecture is nothing but the Bogomolov-type analogue of this conjecture, which was also proposed as \cite[Conjecture 1.2]{Dimitrov2020a}. We also remark that it overlaps with a conjecture already proposed in Zhang's 1998 ICM talk \cite[Section 4]{Zhang1998a}. Throughout this article, the subfield $K \subset \IQbar \subset \IC$ is a number field and $S$ is an irreducible algebraic variety over $K$. The variety $S$ serves as the base of a family $\pi: A \rightarrow S$ of abelian varieties. Furthermore, we assume being given an immersion $\iota: A \hookrightarrow \IP^N_K$ into projective space and a Weil height $h_{\mathcal{O}(1)}$ associated with the ample line bundle $\caO(1)$ on $\IP^N_K$.\footnote{Not every family of abelian varieties $\pi: A \rightarrow S$ admits a projective immersion even if $S$ does (compare \cite[Chapter XII]{Raynaud1970}), but we can always find a proper closed subset $Z \subsetneq S$ such that $A \setminus \pi^{-1}(Z)$ is quasi-projective (e.g.\ by spreading out from the generic point of $S$), take an immersion $\iota : A \setminus \pi^{-1}(Z) \hookrightarrow \IP^N_K$, and obtain again a canonical height function $\widehat{h}_\iota: (A \setminus \pi^{-1}(Z))(\IQbar) \rightarrow \IR^{\geq 0}$. Theorem \ref{theorem:bogomolov} still makes sense in this setting as (RBC) is invariant under passing to Zariski-dense open subsets of $S$. In particular, the Manin-Mumford part of the theorem holds for general families $\pi: A \rightarrow S$ even without the existence of a projective immersion.} For each closed point $x \in A$, we set
\begin{equation*}
\label{equation::nerontateheight}
\hhat_{\iota}(x) = \lim_{k \rightarrow \infty} \left(\frac{ h_{\mathcal{O}(1)}(\iota \circ [n^k](x))}{n^{2k}} \right).
\end{equation*}
As $[n]$ preserves the proper fibers of $\pi$, this is just the ordinary Néron-Tate height of $x$ with respect to (the symmetric part of) the line bundle $\iota^\ast \caO(1)|_{A_{\pi(x)}}$ on the abelian variety $A_{\pi(x)}$.

\textbf{Relative Bogomolov Conjecture (RBC).} \textit{Let $X$ be an irreducible subvariety $X \subset A$ such that $\pi(X)=S$. Assume that $X$ is not a subvariety of codimension $\leq \dim(S)$ in any horizontal torsion coset $Y \subset A$. Then, there exists some $\varepsilon(X)>0$ such that the set $$\{ \text{closed point } x \in X \ | \ \hhat_\iota(x) < \varepsilon(X) \}$$ is not Zariski-dense.}

The notion of horizontal torsion coset, which morally is the analogue of an abelian subvariety translated by a torsion point, demands a formal definition: Let $S^\prime \rightarrow S$ be a generically finite map and $\tau: S^\prime \rightarrow A_{S^\prime}$ a torsion section of the base change $\pi_{S^\prime}: A_{S^\prime} \rightarrow S^\prime$. For each subvariety $X \subseteq A_{S^\prime}$, we define its translate $X + \tau$ to be the image of $X \times_{S^\prime} \tau(S^\prime)$ under the (fiberwise) addition $A_{S^\prime} \times_{S^\prime} A_{S^\prime} \rightarrow A_{S^\prime}$. An irreducible variety $X \subseteq A$ is called a horizontal torsion coset if there exists a generically finite map $S^\prime \rightarrow S$, an $S^\prime$-flat subgroup scheme $B \subseteq A_{S^\prime}$, and a torsion section $\tau: S^\prime \rightarrow A_{S^\prime}$ such that the translate $B + \tau \subseteq A_{S^\prime}$ projects onto $X$. 

It should be noted that (RBC) is independent of the chosen immersion $\iota: A \hookrightarrow \IP^N_K$, which is not completely trivial as $\hhat_{\iota}$ appears in its statement. Let $\iota,\iota^\prime: A \hookrightarrow \IP^N_K$ be two projective immersions and write $\eta$ for the generic point of $S$. Then there exists a positive integer $k$ such that both $(\iota^\ast \caO(1)^{\otimes k} \otimes (\iota^\prime)^\ast \caO(1)^{\otimes -1})|_\eta$ and $(\iota^\ast \caO(1)^{\otimes -1} \otimes (\iota^\prime)^\ast \caO(1)^{\otimes k})|_{\eta}$ are ample. There exists thus an open dense subset $U \subseteq S$ such that both $(\iota^\ast \caO(1)^{\otimes k} \otimes (\iota^\prime)^\ast \caO(1)^{\otimes -1})|_s$ and $(\iota^\ast \caO(1)^{\otimes -1} \otimes (\iota^\prime)^\ast \caO(1)^{\otimes k})|_{s}$ are ample for all $s \in U$. Since $\hhat_{\iota}$ and $\hhat_{\iota^\prime}$ restrict to the usual Néron-Tate heights on fibers, it follows that
\begin{equation*}
k^{-1} \cdot \hhat_\iota(x) \leq \hhat_{\iota^\prime}(x) \leq k \cdot \hhat_{\iota}(x)
\end{equation*}
for all closed points $x \in \pi^{-1}(U)$. As it clearly suffices to prove (RBC) for the restriction $X|_U \subseteq A|_U$, this shows the independence of (RBC) from the chosen immersion $\iota$.



The main result of our article concerns (RBC), and is a generalization of \cite[Theorem 1.4]{DeMarco2020}. As usual, (RBC) implies Manin-Mumford type results in the same relative settings. It has been already mentioned that the Manin-Mumford analogue of (RBC) was proposed by Pink \cite[Conjecture 6.2]{Pink2005a}. For $\dim(X)=1$, results related to Pink's conjecture have been obtained by Masser and Zannier \cite{Masser2014}, but no result of relative Manin-Mumford type seems to have been known for subvarieties $X \subset A$ of relative dimension $>1$ up to now.\footnote{While this article was in revision, Gao and Habegger have announced a general proof of the relative Manin-Mumford conjecture.} It should be remarked that \cite{DeMarco2020} uses the relative Manin-Mumford conjecture \cite{Masser2012} in order to prove (RBC). While this article was in preparation, DeMarco and Mavraki \cite{DeMarco2022c} were able to remove this dependence and to generalize their previous work; an essential new ingredient is the separation of holomorphic and anti-holomorphic terms that is also used here (compare our Section \ref{subsection::separation} below with the first step in \cite[Subsection 8.2]{DeMarco2022c}) and which has been first introduced by André, Corvaja, and Zannier (see \cite[Subsection 5.2]{Andre2020}). Our proof, which includes the case considered in \cite{DeMarco2020}, also avoids a dependence on the relative Manin-Mumford conjecture so that we obtain it instead as a genuine corollary in all cases under consideration. For general curves $X \subset A$ defined over $\IQbar$, the relative Manin-Mumford conjecture has been proven recently by Masser and Zannier \cite[Theorem 1.7]{Masser2020}.

\begin{theorem}
	\label{theorem:bogomolov}
	(RBC) is true if $A$ is the fibered product $E_1 \times_S E_2 \times_S \cdots \times_S E_g$ of families of elliptic curves $E_i \rightarrow S$ ($1 \leq i \leq g$) over a base variety $S$.
\end{theorem}

The reader may note that we include the case of isotrivial families in the theorem. Furthermore, Theorem \ref{theorem:bogomolov} immediately implies (RBC) in the slightly more general situation that the fiber $A|_{\eta}$ over the generic point $\eta$ of $S$ is isogeneous to a fibered product of families of elliptic curves. 
The proof of Theorem \ref{theorem:bogomolov} constitutes the content of Sections \ref{section::first} to \ref{section::last}. For further details on the structure of the proof, we refer the reader to Section \ref{section::overview}. 

It is a reasonable first guess that the equidistribution result \cite[Theorem 1]{Kuehnea} implies (RBC) in general just as the classical Bogomolov conjecture can be proven by means of equidistribution. Unfortunately, the analogy with the classical case leads one astray here. The Ullmo-Zhang approach \cite{Ullmo1998,Zhang1998} to the Bogomolov conjecture does not transfer well to the relative setting because the Faltings-Zhang map
\begin{equation*}
A^n \longrightarrow A^{n-1}, \ (x_1,x_2,\dots,x_n) \longmapsto (x_1-x_2,\dots,x_{n-1}-x_n),
\end{equation*}
has only an $S$-fibered analogue that is too weak for a reproduction of the arguments used in \cite[Section 4]{Zhang1998}. In short, we cannot subtract points in two different fibers as there is no group structure on the total space. One can still subtract points contained in the same fiber, and this gives rise to the uniform results obtained in \cite{Dimitrov2021, Dimitrov2021a, Kuehnea}. Here, we can prove our Theorem \ref{theorem:bogomolov} by making use of the additional product structure available on $E_1 \times_S \cdots \times_S E_g$, but our argument definitely breaks down for generically simple families $A \rightarrow S$. Besides the author's recent equidistribution result, essential tools for the proof of Theorem \ref{theorem:bogomolov} are André's theorem \cite{Andre1992} on the normality of the monodromy group in admissible variations of mixed Hodge structures and the Ax-Schanuel conjecture for mixed Shimura varieties proven by Gao \cite{Gao2018}.

Finally, let us mention that by the argument given in \cite{Dimitrov2020a}, one can easily see that (RBC) implies a uniform version of the Bogomolov conjecture for curves of arbitrary genus $g \geq 2$ whose Jacobian is a product of elliptic curves. In other words, one can partially recover the author's previous result \cite[Theorem 2]{Kuehnea}. Likewise, the result of DeMarco, Krieger, and Ye \cite{DeMarco2020a}, which answered a question of Bogomolov and Tschinkel \cite{Bogomolov2018}, can be already deduced from the case of (RBC) proven here. However, the results of \cite{Kuehnea} and their improvement by Yuan \cite{Yuan2021} present substantially more general cases, in which (RBC) remains widely open.

\textbf{Notation and conventions.} 
\textit{Algebraic Geometry (General).} Denote by $k$ an arbitrary field. A \textit{$k$-variety} is a reduced separated scheme of finite type over $k$. By a \textit{subvariety} of a $k$-variety we mean a reduced closed subscheme. A subvariety is determined by its underlying topological space and we frequently identify both. 
Furthermore, $X^{\mathrm{sm}}$ denotes the smooth locus of $X$. 



\textit{Generic sequences.} Let $X$ be an algebraic $k$-variety. If $X$ is irreducible, we say that a sequence $(x_i) \in X^\IN$ of closed points is \textit{$X$-generic} if none of its subsequences is contained in a proper algebraic subvariety of $X$. Note that a sequence is $X$-generic if and only if it converges to the generic point of $X$ in the Zariski topology. If the irreducible variety $X$ can be inferred from context, we simply say \textit{generic} instead of $X$-generic. 



\textit{Continuity and smoothness.} 
We use $\mathscr{C}^0$ 
as an abbreviation for continuous
. For any topological space $X$, $\mathscr{C}^0(X)$ denotes the real-valued continuous functions on $X$ and $\mathscr{C}^0_c(X)$ the real-valued continuous functions on $X$ having compact support. 



\textit{Analytification.} For a number field $K \subset \IC$ and a $K$-variety $X$, we write $X(\IC)$ for the complex analytic space associated with $X_{\IC}$.

\textit{Complex spaces.} Let $M$ be a reduced complex (analytic) space (e.g., the analytic space $X(\IC)$ associated with a $K$-variety $X$). Recall that this means that $M$ is locally biholomorphic to a closed analytic subvariety $V$ in a complex domain $U \subset \IC^n$. 
A \textit{$\mathscr{C}^\infty$-form} $\omega$ on $M$ is a differential form on the smooth locus $M^\mathrm{sm}$ of $M$ with the following extension property: $M$ can be covered by local charts $V \subset U \subset \IC^n$ as above such that for each chart the differential form $\omega|_{V^{\mathrm{sm}}}$ is the restriction of a $\mathscr{C}^\infty$-differential form on $U$. There are also well-defined linear operators $d$, $\del$, $\delbar$ on the $\mathscr{C}^\infty$-differential forms on $M$. For each local chart $V \subset U \subset \IC^n$, these are simply the restrictions of the operators of the same name on $\IC^n$. 

\textit{Moduli spaces of elliptic curves}.  We write $Y(\mathcal{N})$ for the moduli stack over $\IQ$ parameterizing elliptic curves with level $\mathcal{N}$ structure (\cite[Section 13.1]{Olsson2016}). This is a smooth quasi-projective variety if $\mathcal{N}\geq 3$ (\cite[Corollary 4.7.2]{Katz1985}). For each $\mathcal{N}\geq 1$, we write $\xi_\mathcal{N}: \mathcal{E}(\mathcal{N}) \rightarrow Y(\mathcal{N})$ for the universal family of elliptic curves with level $\mathcal{N}$ structure.

\textit{Siegel upper half-space}. We write $\mathcal{H}_g$ for the Siegel upper half-space of degree $g$, considered as a complex manifold.

\section{Setting-up the proof of Theorem \ref{theorem:bogomolov}}
\label{section:bogomolov}
\label{section::first}

As in the statement of the theorem, let $S$ be a base variety, let $E_j \rightarrow S$ ($1\leq j \leq g$) be families of elliptic curves, and set $$\pi: A = E_1 \times_S \cdots \times_S E_g \rightarrow S.$$ Furthermore, let $X \subseteq A$ be a subvariety of dimension $d$, for which we want to prove (RBC). We start by making some additional assumptions for the proof of the theorem in Section \ref{subsection::reductions} and introduce coordinates in Section \ref{subsection::coverings}. Following this, we give an overview of the main argument in Section \ref{section::overview}.

\section{Reductions}
\label{subsection::reductions}
(i) In our proof, we suppose that the conclusion of (RBC) is false and show that its main assumption cannot hold under this assumption. This means our goal is to show that there exists a horizontal torsion coset $Y \subseteq A$ such that $X$ is a subvariety of codimension $\leq \dim(S)$ in $Y$. In the sequel, we can hence work with an $X$-generic sequence $(x_i) \in X^\IN$ such that $\hhat_{\iota}(x_i) \rightarrow 0$. 

\vspace{0.2cm}

(ii) 
We can assume that $g = d+1$. In fact, if 
\begin{equation*}
\dim(A) - \dim(S) = g < d + 1 = \dim(X) + 1, 
\end{equation*}
then $\mathrm{codim}_A(X) < \dim(S) + 1$ so that the assumption of (RBC) is not satisfied. If $g > d + 1$, we choose a projection 
\begin{equation*}
\varphi: A \rightarrow E_{j_1} \times E_{j_2} \times  \cdots \times E_{j_{d+1}}, \ 1 \leq j_1 < j_2 < \cdots < j_{d+1}\leq g.
\end{equation*}
It clearly suffices to prove (RBC) for $\varphi(X)$.

(iii\textsubscript{0}) We can also make the following assumption: For any fibered product 
\begin{equation*}
\pi^\prime: A^\prime = E_{1}^\prime \times_{S^\prime} E_{2}^\prime \times_{S^\prime} \cdots \times_{S^\prime} E_{g^\prime}^\prime \rightarrow S^\prime
\end{equation*}
of families of elliptic curves $E_j^\prime \rightarrow S^\prime$ ($1\leq j\leq g^\prime \leq g$) and any commutative diagram
\begin{equation}
\label{equation::new_family}
\begin{tikzcd}
A \ar[r, "\varphi", two heads] \ar[d, "\pi"]
& A^\prime \ar[d, "\pi^\prime"] \\
S \ar[r, two heads] & S^\prime
\end{tikzcd}
\end{equation}
with $\varphi$ a fiberwise homomorphism, we have 
\begin{equation}
\label{equation::dimension_decent}
\dim(\varphi(X)) \geq \dim(X) - (g - g^\prime)
\end{equation}
with equality if and only if $\dim(A)=\dim(A^\prime)$. 

In fact, the special case $\dim(A)=\dim(A^\prime)$ is trivial as then $\varphi$ is a fiberwise isogeny and $g=g^\prime$. Thus, we may suppose that there exist such $\pi^\prime: A^\prime \rightarrow S^\prime$ and $\varphi: A \rightarrow A^\prime$ with $\dim(A^\prime)<\dim(A)$ and
\begin{equation}
\label{equation::violation}
\dim(\varphi(X)) \leq \dim(X) - (g - g^\prime).
\end{equation}
By an induction on $\dim(A)$, we can also assume that (RBC) is already proven for the family $\pi^\prime: A^\prime \rightarrow S^\prime$. The sequence $\varphi(x_i)$ is $\varphi(X)$-generic and satisfies $\hhat_{\iota^\prime}(\varphi(x_i)) \rightarrow 0$ for any immersion $\iota^\prime: A^\prime \hookrightarrow \IP^{N^\prime}$. Therefore $\varphi(X)$ has to violate the assumption in (RBC), which means that there exists a horizontal torsion coset $Y^\prime \subseteq A^\prime$ containing $\varphi(X)$ and satisfying
\begin{equation*}
\dim(Y^\prime) - \dim(\varphi(X)) \leq \dim(S^\prime).
\end{equation*}
Using the inequality \eqref{equation::violation}, we infer further that
\begin{equation*}
\dim(Y^\prime) - \dim(X) + (g - g^\prime) \leq \dim(S^\prime).
\end{equation*}
The irreducible components of $\varphi^{-1}(Y^\prime)$ are horizontal torsion cosets. We can pick such an irreducible component $Y$ containing $X$ and notice that
\begin{equation*}
\dim(Y) \leq \dim(Y^\prime) + \dim(S) - \dim(S^\prime) + (g - g^\prime).
\end{equation*}
Combining the last two inequalities, we obtain
\begin{equation*}
\dim(Y) - \dim(X) \leq \dim(S).
\end{equation*}
This is a violation of the assumption in (RBC) for $X \subseteq A$, and thus there is nothing left to prove.

(iii) Before continuing with our reductions, let us point out some consequences of the previous assumption. For each $1 \leq k \leq g$, we write $$\mathrm{pr}_{\widehat{k}}: \textstyle\prod_{j=1}^g E_j \longrightarrow \textstyle\prod_{j=1, j \neq k}^g E_j$$ for the standard projection. The images $\pr_{\widehat{k}}(X)$, $1 \leq k \leq g$, have then all dimension $\dim(X)$. In fact, we have a commutative square
\begin{equation*}
\begin{tikzcd}
A = \prod_{j=1}^g E_j \ar[r, "\mathrm{pr}_{\widehat{k}}", two heads] \ar[d]
& \prod_{j=1, j \neq k}^g E_j \ar[d] \\
S \ar[r, equal] & S 
\end{tikzcd}
\end{equation*}
as in \eqref{equation::new_family} above so that \eqref{equation::dimension_decent} implies
\begin{equation*}
\dim(\pr_{\widehat{k}}(X)) > \dim(X) - 1.
\end{equation*}
and thus $\dim(\pr_{\widehat{k}}(X)) = \dim(X)$.

A further consequence of the previous reduction is that $X$ is non-degenerate. In fact, if $X$ were degenerate, then \cite[Theorem 1.1 (i)]{Gao2018a} would produce a new family $\pi^\prime: A^\prime \rightarrow S^\prime$ and a fiberwise homomorphism $\varphi: A \twoheadrightarrow A^\prime$ filling a diagram \eqref{equation::new_family} but violating the constraint \eqref{equation::dimension_decent} on the dimension of $\varphi(X)$. 

Even more, all the images $\mathrm{pr}_{\widehat{k}}(X)$, $1 \leq k \leq g$, are non-degenerate as well. Indeed, assume that this is not the case. Then, \cite[Theorem 1.1 (i)]{Gao2018a} supplies a commuting diagram
\begin{equation*}
\begin{tikzcd}
A = \prod_{j=1}^g E_j \ar[r, "\mathrm{pr}_{\widehat{k}}", two heads] \ar[d]
& \prod_{j=1, j \neq k}^g E_j \ar[d] \ar[r, "\varphi"] & \prod_{j=1}^{g^\prime} E_j^\prime \ar[d] \\
S \ar[r, equal] & S \ar[r] & S^\prime
\end{tikzcd}
\end{equation*}
with fiberwise homomorphisms along the upper row and such that
\begin{equation*}
\dim(\varphi(\pr_{\widehat{k}}(X))) < \dim(\pr_{\widehat{k}}(X)) - (g - 1 - g^\prime).
\end{equation*}
Using our assumption, we deduce conversely from \eqref{equation::dimension_decent} that
\begin{equation*}
\dim(\varphi(\pr_{\widehat{k}}(X))) > \dim(X) - (g - g^\prime).
\end{equation*}
A combination of these two inequalities yields
\begin{equation*}
\dim(\pr_{\widehat{k}}(X)) - (g - 1 - g^\prime) \geq \dim(X) - (g - g^\prime) + 2,
\end{equation*}
which is equivalent to the absurdity 
\begin{equation*}
\dim(X) = \dim(\pr_{\widehat{k}}(X)) \geq \dim(X) + 1.
\end{equation*}
This proves our claim about the non-degeneracy of $\pr_{\widehat{k}}(X)$, $1 \leq k \leq g$.

For the convenience of the reader, let us briefly summarize the assumptions that we can and do tacitly assume in the following:

\begin{enumerate}	
	\item[(i)]	there exists an $X$-generic sequence $(x_i) \in X^\IN$ such that $\hhat_{\iota}(x_i) \rightarrow 0$,
	\item[(ii)] $\dim(A) - \dim(S) = g = d+1 = \dim(X) + 1$,
	\item[(iii)] $X$ is non-degenerate, and every projection $\pr_{\widehat{k}}(X)$, $1\leq k \leq g$, is non-degenerate and has dimension $\dim(X)$.
\end{enumerate}

We continue with imposing two further restrictions on the family $\pi: A \rightarrow S$. 

(iv) First, we can assume without loss of generality that it is a subfamily of the $g$-fold self-product of the universal family $\xi_\mathcal{N}: \mathcal{E}(\mathcal{N}) \rightarrow Y(\mathcal{N})$, $\mathcal{N} \geq 3$. There exists a classifying map $\mathcal{c}: S \rightarrow Y(1)^g$ (of algebraic stacks) such that
\begin{equation*}
\begin{tikzcd}
A \ar[r, "\varphi"] \ar[d] & \mathcal{E}(1)^g \ar[d] \\
S \ar[r, "\mathcal{c}"] & Y(1)^g
\end{tikzcd}
\end{equation*}
is a cartesian square, and we consider its pullback
\begin{equation*}
\begin{tikzcd}
A^\prime = A \times_{S} S^\prime \ar[r, "\varphi^\prime"] \ar[d] & \mathcal{E}(\mathcal{N})^g \ar[d] \\
S^\prime = S \times_{Y(1)} Y(\mathcal{N}) \ar[r, "\mathcal{c}^\prime"] & Y(\mathcal{N})^g
\end{tikzcd}
\end{equation*}
along $Y(\mathcal{N}) \rightarrow Y(1)$. As $S^\prime \rightarrow S$ is finite, it clearly suffices to prove (RBC) for the subvariety $X^\prime = X \times_S S^\prime \subseteq A^\prime$. Consider the induced subfamily $\pi^\prime: \varphi^\prime(A) \rightarrow \mathcal{c}^\prime(S)$. Assuming that (RBC) holds for this subfamily, we obtain as above that there exists a horizontal torsion coset $Y^\dprime \subseteq \varphi^\prime(A^\prime)$ with the property that
\begin{equation*}
\dim(Y^\dprime) - \dim(\varphi^\prime(X^\prime)) \leq \dim(\mathcal{c}^\prime(S^\prime)).
\end{equation*}
The preimage $Y^\prime=(\varphi^\prime)^{-1}(Y^\dprime)$ is evidently a horizontal torsion coset of dimension $\dim(Y^\dprime) + \dim(S^\prime) - \dim(\mathcal{c}(S^\prime))$ containing $X$, and furthermore
\begin{equation*}
\dim(Y^\prime)-\dim(X^\prime) \leq \dim(Y^\dprime) + \dim(S^\prime) - \dim(\mathcal{c}^\prime(S^\prime)) - \dim(\varphi^\prime(X^\prime)) \leq \dim(S^\prime)
\end{equation*} 
this is again contradicting the assumption of (RBC) for $X^\prime \subset A^\prime$ and hence also for $X \subset A$. We can and do hence assume that $A= \mathcal{E}(\mathcal{N})^g|_S$ for some $S \subseteq Y(\mathcal{N})^g$ where $\mathcal{N} \geq 3$ is fixed once and for all in the sequel.

(v) Second, we can use our free choice of the immersion $\iota: A \hookrightarrow \IP^N_K$ in the statement of (RBC) to guarantee that
\begin{equation*}
\iota = (\sigma \circ (\iota_0 \times \cdots \times \iota_0))|_S
\end{equation*}
where $\iota_0: \mathcal{E}(\mathcal{N}) \hookrightarrow \IP^{N_0}_K$ is an arbitrary projective immersion such that $\iota_0^\ast \caO(1)$ is fiberwise symmetric and $\sigma: \IP^{N_0}_K \times \cdots \times \IP^{N_0}_K \hookrightarrow \IP^N_K$ is the Segre embedding.

\section{Coverings and coordinates}
\label{subsection::coverings}

We recall the following commutative diagram, whose horizontal rows are universal coverings: \tikzcdset{every label/.append style = {font = \normalsize}}
\begin{equation*}
\begin{tikzcd}
\arrow[loop left, "{\normalsize \Gamma = (\IZ^2 \rtimes \Gamma(\mathcal{N}))^g}"] (\mathbb{C} \times \mathcal{H}_1)^g \ar[rr, "\mathcal{u}_{\mathrm{mixed}}"] \ar[d, two heads, "\ \widetilde{\pi}"]
& & \mathcal{E}(\mathcal{N})^g(\IC) \ar[d, two heads, "\ \xi_{\mathcal{N}} \times \cdots \times \xi_{\mathcal{N}}"] 
\\
\arrow[loop left, "{\normalsize \Gamma(\mathcal{N})^g}"] \mathcal{H}_1^g  \ar[rr, "\mathcal{u}_{\mathrm{pure}}"] & & Y(\mathcal{N})^g(\IC).
\end{tikzcd}
\end{equation*}
The covering transformations of $\mathcal{u}_{\mathrm{pure}}$ (resp.\ $\mathcal{u}_{\mathrm{mixed}}$) are given by the group $\Gamma(\mathcal{N})^g$ (resp.\ $(\IZ^2 \rtimes \Gamma(\mathcal{N}))^g$) where
\begin{equation*}
\Gamma(\mathcal{N}) = \left\{ \begin{pmatrix} a & b \\ c & d \end{pmatrix} \in \SL_2(\IZ) \ \middle| \ 
\begin{pmatrix} a & b \\ c & d \end{pmatrix}
\equiv
\begin{pmatrix} 1 & 0 \\ 0 & 1 \end{pmatrix} (\text{mod $\mathcal{N}$})
\right\}
\end{equation*}
acts on $\IZ^2$ through its standard representation. This identification is such that the group element
\begin{equation*}
\left(
\begin{pmatrix} m_1 \\ n_1 \end{pmatrix},
\begin{pmatrix} a_1 & b_1 \\ c_1 & d_1 \end{pmatrix},
\dots, 
\begin{pmatrix} m_g \\ n_g \end{pmatrix}, 
\begin{pmatrix} a_g & b_g \\ c_g & d_g \end{pmatrix} \right) 
\in (\IZ^2 \rtimes \Gamma(\mathcal{N}))^g
\end{equation*}
sends $(z_1,\dots,z_g,\tau_1,\dots,\tau_g) \in (\mathbb{C} \times \mathcal{H}_1)^g$ to 
\begin{equation*}
\left(\frac{z_1+m_1+n_1\tau_1}{c_1\tau_1+d_1},\dots, \frac{z_g+m_g+n_g\tau_g}{c_g\tau_g+d_g}, \frac{a_1\tau_1+b_1}{c_1\tau_1+d_1}, \dots, \frac{a_g\tau_g+b_g}{c_g\tau_g+d_g} \right)
\end{equation*}
(compare e.g.\ \cite[Section 8.8]{Birkenhake2004}).


On the complex manifold $(\mathbb{C} \times \mathcal{H}_1)^g$, we have the global holomorphic standard coordinates $z_l, \tau_l$ ($1\leq l \leq g$) and their complex-conjugates $\overline{z}_l, \overline{\tau}_l$. In addition, we define $2g$ real-analytic functions
\begin{equation*}
x_l, y_l: (\mathbb{C} \times \mathcal{H}_1)^g \longrightarrow \IR, \ 1 \leq l \leq g,
\end{equation*}
by demanding $z_l = x_l + \tau_l y_l$. 

The preimage $\mathcal{u}_{\mathrm{pure}}^{-1}(S)$ (resp.\ $\mathcal{u}_{\mathrm{mixed}}^{-1}(X)$) decomposes into irreducible analytic components of dimension $\dim(S)$ (resp.\ $\dim(X)$), on which the group $\Gamma(\mathcal{N})^g$ (resp.\ $(\IZ^2 \rtimes \Gamma(\mathcal{N}))^g$) acts transitively. Due to the absence of elliptic fixed points on $Y(\mathcal{N})$ (see e.g.\ \cite[Exercise 2.3.7]{Diamond2005}), the map $\mathcal{u}_{\mathrm{pure}}$ (resp.\ $\mathcal{u}_{\mathrm{mixed}}$) is étale and these components coincide with the connected components of $\mathcal{u}_{\mathrm{pure}}^{-1}(S)$ (resp.\ $\mathcal{u}_{\mathrm{mixed}}^{-1}(X)$) in the euclidean topology. In the sequel, we keep fixed an irreducible component $\widetilde{X}$ of $X$. Its image $\widetilde{S}=\widetilde{\pi}(\widetilde{X})$ is then an analytic component of the preimage of $S$.

\section{Overview of the proof}
\label{section::overview}

For convenience of the reader, we briefly expose the main lines of the argument employed for the proof of Theorem \ref{theorem:bogomolov} in the following sections. In Section \ref{subsection::equidistribution}, we use the product structure and the equidistribution results of \cite{Kuehnea} to obtain differential-geometric conditions on $\widetilde{X}$. These conditions appear as real-analytic differential equations \eqref{equation::differential_equation_1}, which can be written down explicitly in local charts of $\widetilde{X}$ and the local coordinates introduced in Section \ref{subsection::coverings} above. 

A natural way to exploit these is to use monodromy, to wit, the fact that the stabilizer $\mathrm{Stab}(\widetilde{X}) \subseteq (\IZ^2 \rtimes \Gamma(\mathcal{N}))^g$ is rather large. This largeness follows by Hodge-theoretic techniques, which are exposed in Sections \ref{section::hodge1} to \ref{subsection::monodromy}. Besides rather explicit computations, we make use of a theorem of André \cite[Theorem 1]{Andre1992} on the normality of the algebraic monodromy group. Unfortunately, the equations \eqref{equation::differential_equation_1} are invariant under monodromy so that a direct application of monodromy fails; this failure is not really surprising since the Betti form, which gives rise to these equations, is invariant under monodromy. 

To the rescue comes an important idea of André, Corvaja, and Zannier \cite[Subsection 5.2]{Andre2020} that allows us to actually take advantage of the fact that the real-analytic equations \eqref{equation::differential_equation_1} contain both holomorphic and anti-holomorphic terms. In short, we replace $\widetilde{X}$ and \eqref{equation::differential_equation_1} with $\widetilde{X} \times \widetilde{X}$ and a new set of real-analytic differential equations \eqref{equation::differential_equation_2}. This is the content of Section \ref{subsection::separation}.

Following a computation of the transformation behavior of \eqref{equation::differential_equation_2} under monodromy (Section \ref{section::monodromy}) and a final technical preparation in Section \ref{section::technical}, we use explicit elements of the algebraic monodromy group to prove first that -- under the assumptions of (RBC) -- all factors are isogeneous (Sections \ref{section::reduction1} and \ref{subsection::diagonal}). With this at our disposal, we then deduce a linear equation \eqref{equation::linear_equation} on $\widetilde{X}$ in the coordinates $z_1,\dots,z_n$. Gao's mixed Ax-Schaunel theorem \cite{Gao2018} enables us then to conclude the proof of Theorem \ref{theorem:bogomolov} (Section \ref{section::last}).

\section{Equidistribution}
\label{subsection::equidistribution}

Let us start with defining the equilibrium measure. The $(1,1)$-form
\begin{equation}
\label{equation::invariant_form}
\frac{i}{\mathrm{Im}(\tau_j)} (dz_j - y_j d\tau_j ) \wedge (d\overline{z}_j - y_j d\overline{\tau}_j)
\end{equation}
on $(\mathbb{C} \times \mathcal{H}_1)^g$ is $(\IZ^2 \rtimes \Gamma(\mathcal{N}))^g$-invariant (see \cite[Lemma 2.6]{Dimitrov2021}) and hence descends to a $(1,1)$-form $\alpha_j$ on $A(\IC)$. We define the $(1,1)$-form
\begin{equation*}
\beta = \sum_{j=1}^g \alpha_i
\end{equation*}
and consider its $d$-fold exterior power
\begin{equation}
\label{equation:formula_mu}
\beta^{\wedge d} = d! \cdot \sum_{j=1}^{g} \alpha^\prime_j, \ \alpha^\prime_j = 
\alpha_1 \wedge \cdots \wedge \alpha_{j-1} \wedge \alpha_{j+1} \wedge \cdots \wedge \alpha_g.
\end{equation}
Up to multiplication with some (strictly) positive real number, this coincides with the smooth closed $(1,1)$-form on $A(\IC)$ provided by \cite[Lemma 2.6]{Dimitrov2021a} (compare also the proofs of \cite[Lemmas 7 and 11]{Kuehnea}). By Theorem \cite[Theorem 1]{Kuehnea} (and its proof), there exists some $\mathcal{k}_X>0$ such that 
\begin{equation}
\label{equation:equidistribution_bogo1}
\frac{1}{\# \mathbf{O}(x_i)}\sum_{x \in \mathbf{O}(x_i)} f(x)\longrightarrow \mathcal{k}_{X} \int_{X(\IC)} f \beta^{\wedge d}, 
\ i \rightarrow \infty,
\end{equation}
for every continuous function $f \in \mathscr{C}^0_c(X)(\IC)$. Note that each $\alpha_j^\prime|_{X(\IC)}$, $1\leq j \leq g$, is non-zero as $\pr_{\widehat{j}}(X)$ is non-degenerate by the reductions in Section \ref{subsection::reductions} above.

For given integers $n_1,\dots,n_{g} > 0$, we consider the homomorphism
\begin{equation*}
\varphi: E_1 \times_S \cdots \times_S E_{g} \longrightarrow E_1 \times_S \cdots \times_S E_{g}, \ (x_1,\dots,x_{g}) \longmapsto ([n_1]x_1,\dots,[n_{g}]x_{g}),
\end{equation*}
and set $Y=\varphi(X)$. We claim that $(\varphi^\ast \beta|_Y)^{\wedge d}$ is real-proportional to $(\beta|_X)^{\wedge d}$. As $\varphi$ is étale, there exists a non-empty open subset $U \subseteq X(\IC)$ such that the restriction $\varphi|_U: U \rightarrow \varphi(U)$ is a biholomorphism. Writing $y_i = \varphi(x_i)$, the sequence $(y_i)$ is $Y$-generic and satisfies $\hhat_\iota(y_i)\rightarrow 0$. By the above argument applied to $Y$ instead of $X$, we have hence
\begin{equation}
\label{equation:equidistribution_bogo2}
\frac{1}{\# \mathbf{O}(y_i)}\sum_{y \in \mathbf{O}(y_i)} g(y)\longrightarrow \mathcal{k}_Y \int_{Y(\IC)} g \beta^{\wedge d}, 
\ i \rightarrow \infty,
\end{equation}
for every $g\in \mathcal{C}^0_c(Y(\IC))$. For any continuous function $f\in \mathcal{C}^0_c(U)$, there exists a (unique) continuous function $g\in \mathcal{C}^0_c(\varphi(U))$ such that $f = g \circ \varphi$. As $\varphi(\mathbf{O}(x_i))=\mathbf{O}(y_i)$, we have 
\begin{equation*}
\frac{1}{\# \mathbf{O}(x_i)}\sum_{x \in \mathbf{O}(x_i)} f(x) = \frac{1}{\# \mathbf{O}(y_i)}\sum_{y \in \mathbf{O}(y_i)} g(y)
\end{equation*}
in this situation. Hence the limits in \eqref{equation:equidistribution_bogo1} and \eqref{equation:equidistribution_bogo2} are equal, which means that
\begin{equation*}
\mathcal{k}_X \int_{X(\IC)} f \beta^{\wedge d} =
\mathcal{k}_Y \int_{X(\IC)} f (\varphi^\ast \beta)^{\wedge d}
\end{equation*}
for any $f\in\mathcal{C}^0_c(V)$. Varying the test function $f$, we infer that $\mathcal{k}_X (\beta|_V)^{\wedge d} = \mathcal{k}_Y (\varphi^\ast \beta|_V)^{\wedge d}$. Since $\beta$ has real-analytic coefficients, this completes the proof of the claim.

As $\varphi^\ast \alpha_j = n_j^2 \alpha_j$, we have $$\varphi^\ast \alpha_j^\prime = (\textstyle\prod_{\substack{k \in \{1,\dots, g\}, k \neq j}} n_k^2) \cdot \alpha_j^\prime$$ for every $j \in \{1,\dots, g\}$. Thus, we have
\begin{equation*}
(\varphi^\ast \beta|_{X(\IC)})^{\wedge d} = d! \cdot \sum_{j=1}^g (\textstyle\prod_{\substack{k \in \{1,\dots, g\}, k \neq j}} n_k^2) \cdot \alpha_j^\prime|_{X(\IC)}
\end{equation*}
and hence the $(d,d)$-forms
\begin{equation*}
\sum_{j=1}^g \alpha_j^\prime|_{X(\IC)} \ \text{ and } \ \sum_{j=1}^g (\textstyle\prod_{\substack{k \in \{1,\dots, g\}, k \neq j}} n_k^2) \cdot \alpha_j^\prime|_{X(\IC)}
\end{equation*}
are proportional by a positive real constant, which depends on $n_1,\dots,n_g$. 

We claim that all $(d,d)$-forms $\alpha_j^\prime|_{X(\IC)}$, $j \in \{1,\dots, g\}$, are pairwise proportional up to positive real constants. In fact, choosing for example $n_{1} = n_2 = \cdots = n_{g-1} = 1$ and $n_g = 2$ yields that 
\begin{equation*}
\sum_{j=1}^{g} \alpha_j^\prime|_{X(\IC)} \ \text{ and } \ \sum_{j=1}^{g-1} \alpha_j^\prime|_{X(\IC)} + \frac{1}{4} \cdot \alpha_g^\prime|_{X(\IC)}
\end{equation*}
are proportional by a real positive constant. Rewriting this proportionality, we obtain that 
\begin{equation*}
\sum_{j=1}^{g-1} \alpha_j^\prime|_{X(\IC)} \ \text{ and } \ \alpha_g^\prime|_{X(\IC)}
\end{equation*}
are proportional by a positive real constant. (Note that the positivity of the volume forms $\alpha_j^\prime|_{X(\IC)}$, $1\leq j \leq g$, is used here to ensure that these proportionality factors are (strictly) positive.) This implies that also
\begin{equation*}
\sum_{j=1}^{g} \alpha_j^\prime|_{X(\IC)} \ \text{ and } \ \alpha_g^\prime|_{X(\IC)}
\end{equation*}
are proportional by a positive real constant. We obtain similarly that each $\alpha_j^\prime|_{X(\IC)}$,  $j \in \{1,\dots, g -1 \}$, is proportional to $\sum_{j=1}^{g} \alpha_j^\prime|_{X(\IC)}$, whence our claim. For later reference, let us choose reals $r_1,\dots,r_g>0$ such that
\begin{equation*}
r_1 \alpha_1^{\prime}|_{X(\IC)} = r_2 \alpha_2^{\prime}|_{X(\IC)} = \cdots = r_g \alpha_g^{\prime}|_{X(\IC)}.
\end{equation*}

This constitutes a differential-geometric restriction on the analytic subset $X(\IC)$ of $\mathcal{E}(\mathcal{N})(\IC)$. Pulling these back along $\mathcal{u}_{\mathrm{mixed}}$, we obtain similar restrictions on $\widetilde{X} \subset (\IC \rtimes \mathcal{H}_1)^g$. Let us spell these out in terms of a general local chart 
\begin{equation*}
\chi: B_1(0)^d = \{ (w_1,\dots,w_d) \in \IC^d \ | \ \max\{|w_1|,|w_2|,\dots,|w_d|\} < 1 \} \longrightarrow \widetilde{X}.
\end{equation*}
For each function $f$ on $\widetilde{X}$, we simply write $f$ (resp.\ $\partial f/ \partial w_m$, $\partial f/ \partial \overline{w}_m$) instead of $f \circ \chi$ (resp.\ $\partial (f \circ \chi)/\partial w_m$, $\partial (f \circ \chi)/\partial \overline{w}_m$). With this notation, the $(1,1)$-form $(\chi \circ \mathcal{u}_{\mathrm{mixed}})^\ast \alpha_j$, $j \in \{1,\dots, d\}$, on $B_1(0)^d$ equals
\begin{align*}
\frac{i}{\mathrm{Im}(\tau_j)} \left( \sum_{m=1}^d \left[ \frac{\del z_j}{\del w_m} - \frac{\mathrm{Im}(z_j)}{\mathrm{Im}(\tau_j)} \frac{\del \tau_j}{\del w_m} \right] dw_m\right) \wedge
\left( \sum_{m=1}^d \overline{\left[ \frac{\del z_j}{\del w_m} - \frac{\mathrm{Im}(z_j)}{\mathrm{Im}(\tau_j)} \frac{\del \tau_j}{\del w_m} \right]} d\overline{w}_m\right).
\end{align*}
Consequently, the $(d,d)$-form $(\chi \circ \mathcal{u}_{\mathrm{mixed}})^\ast \alpha_j^\prime$, $j \in \{1,\dots, d\}$, on $B_1(0)^d$ equals
\begin{equation}
\label{equation::alpha_j_prime}
\frac{i^d}{\prod_{k \in \{1,\dots,g\},k\neq j} \mathrm{Im}(\tau_k)} 
\left| \det (\mathbf{A}_j) \right|^2 w_1 \wedge \overline{w}_1 \wedge \cdots \wedge w_d \wedge \overline{w}_d
\end{equation}
with
\begin{equation*}
\mathbf A_j = \left( \frac{\del z_l}{\del w_m} - \frac{\mathrm{Im}(z_l)}{\mathrm{Im}(\tau_l)} \frac{\del \tau_l}{\del w_m} \right)_{\substack{l\in \{1,\dots, g\}, l \neq j \\ m \in \{1,\dots, d\}}}.
\end{equation*}
Therefore, the equality $r_j \alpha_j^{\prime}|_X = r_k \alpha_k^{\prime}|_X$, $j,k \in \{1,\dots, g\}$, implies
\begin{equation}
\label{equation::differential_equation_1}
r_j(\tau_j-\overline{\tau}_j)^3 \det(\mathbf B_j) \det(\mathbf C_j) 
= 
r_k(\tau_k-\overline{\tau}_k)^3 \det(\mathbf B_k) \det(\mathbf C_k) 
\end{equation}
on $B_1(0)^d$ where we set
\begin{equation*}
\mathbf B_j = \left( (\tau_l -\overline{\tau}_l) \cdot \frac{\del z_l}{\del w_m} - (z_l - \overline{z}_l) \cdot \frac{\del \tau_l}{\del w_m} \right)_{\substack{l\in \{1,\dots, g\}, l \neq j \\ m \in \{1,\dots, d\}}},
\end{equation*}
and
\begin{equation*}
\mathbf C_j = \left( (\tau_l -\overline{\tau}_l) \cdot \frac{\del \overline{z}_l}{\del \overline{w}_m} - (z_l - \overline{z}_l) \cdot \frac{\del \overline{\tau}_l}{\del \overline{w}_m} \right)_{\substack{l\in \{1,\dots, g\}, l \neq j \\ m \in \{1,\dots, d\}}}
\end{equation*}
for each $j \in \{1,\dots, g\}$.

\section{A variation of mixed Hodge structures on $X(\IC)$}
\label{section::hodge1}

In this section, we decorate the complex analytic space $X(\IC)$ with a variation of mixed $\IZ$-Hodge structures. We refer to \cite[Subsection 14.4.1]{Peters2008} for basic definitions concerning (admissible) variations of mixed $\IZ$-Hodge structures on complex analytic spaces. 

As a starting point, we endow the family $\mathcal{E}(\mathcal{N})(\IC)$ with a variation $H^\prime=(\mathbb{V}_\IZ,\mathcal{F}^\bullet,W_\bullet)$ of mixed $\IZ$-Hodge structures following Deligne \cite[Section 10]{Deligne1974}: The $\IZ$-modules
\begin{equation*}
\mathbb{V}_{\IZ,x} = \{ (l,n) \in \mathrm{Lie}(\mathcal{E}(\mathcal{N})_{\xi_\mathcal{N}(x)})(\IC) \times \IZ \ | \ \exp_{\xi_\mathcal{N}(x)}(l) = [n](x) \}
\end{equation*}
are naturally the stalks of a local system $\mathbb{V}_{\IZ}$ on $\mathcal{E}(\mathcal{N})(\IC)$ having $\IZ$-rank $3$. Furthermore, the Lie group exponential yields an exact sequence
\begin{equation*}
0 \longrightarrow \xi_{\mathcal{N}}^\ast (R^{1} \xi_{\mathcal{N},\ast} \underline{\IZ}_{\mathcal{E}(\mathcal{N})(\IC)})^\vee \longrightarrow \mathbb{V}_{\IZ} \longrightarrow \underline{\IZ}_{\mathcal{E}(\mathcal{N})(\IC)} \longrightarrow 0
\end{equation*}
of $\IZ$-local systems on $\mathcal{E}(\mathcal{N})(\IC)$. (We write $\underline{\IZ}_{\mathcal{E}(\mathcal{N})(\IC)}$ for the locally constant sheaf with stalk $\IZ$ on $\mathcal{E}(\mathcal{N})(\IC)$.) We use this to define the weight filtration as
\begin{equation*}
W_0 = \mathbb{V}_{\mathbb{Z}}, \ W_{-1} =  \xi_{\mathcal{N}}^\ast (R^{1} \xi_{\mathcal{N},\ast} \underline{\IZ}_{\mathcal{E}(\mathcal{N})(\IC)})^\vee, \ W_{-2} = 0.
\end{equation*}
Writing $\mathcal{V} = \mathbb{V}_\IZ \otimes \mathcal{O}_{\mathcal{E}(\mathcal{N})(\IC)}$ for the associated holomorphic vector bundle, we note that the stalk-wise projections $\mathbb{V}_{\IZ,x} \rightarrow \mathrm{Lie}(\mathcal{E}(\mathcal{N})_{\xi_{\mathcal{N}}(x)})(\IC)$ extend to a map 
\begin{equation*}
\phi: \mathcal{V} \rightarrow \xi_{\mathcal{N}}^\ast \mathrm{Lie}(\mathcal{E}(\mathcal{N}))(\IC)
\end{equation*}
of $\caO_{\mathcal{E}(\mathcal{N})(\IC)}$-sheaves where $\mathrm{Lie}(\mathcal{E}(\mathcal{N}))(\IC)$ is the sheaf on $Y(\mathcal{N})(\IC)$ having stalks $\mathrm{Lie}(\mathcal{E}(\mathcal{N})_x)$, $x \in Y(\mathcal{N})(\IC)$. We use this to define the Hodge filtration 
\begin{equation*}
\mathcal{F}^1 = 0, \ \mathcal{F}^0 = \mathrm{ker}(\phi), \ \mathcal{F}^{-1} = \mathcal{V},
\end{equation*}
so that we obtain a mixed Hodge structure of type $\{(0,0), (-1,0), (0,-1)\}$. 

The induced mixed $\IQ$-Hodge structure $H_\IQ=(\mathbb{V}_\IQ,\mathcal{F}^\bullet,W_{\bullet,\IQ})$ can also be described in terms of Shimura theory. For this, we note that $\mathcal{E}(\mathcal{N})(\IC)$ is one of the connected components of the mixed Shimura variety associated to the datum $(\mathbb{G}_{a,\IQ}^2 \rtimes \mathrm{GL}_{2,\IQ}, \IC \times \mathcal{H}_1)$ where $\GL_{2,\IQ}$ acts on $\mathbb{G}_{a,\IQ}^2$ via its standard representation (compare \cite[Chapter 10]{Pink1990}) and the neat open compact subgroup
\begin{equation*}
\mathcal{K}(\mathcal{N}) = \left\{ \left(\begin{pmatrix} m \\ n \end{pmatrix}, \begin{pmatrix} a & b \\ c & d \end{pmatrix} \right) \in \IZhat^2 \times \GL_2(\IZhat) \ \middle| \ 
\begin{pmatrix} a & b \\ c & d \end{pmatrix}
\equiv
\begin{pmatrix} 1 & 0 \\ 0 & 1 \end{pmatrix} (\text{mod $\mathcal{N}$})
\right\};
\end{equation*}
see \cite[Section 0.6]{Pink1990} for the definition of neatness in adelic groups and the proof that $\mathcal{K}(\mathcal{N})$, $\mathcal{N} \geq 3$, is neat. The mixed $\IQ$-Hodge structure $H_\IQ^\prime$ is then induced in the standard way (\cite[Propositions 1.7 and 1.10]{Pink1990}) from the representation 
\begin{equation}
\label{equation::product}
\mathbb{G}_{a,\IQ}^2 \rtimes \mathrm{GL}_{2,\IQ} = 
\begin{pmatrix} \GL_{2,\IQ} & \mathbb{G}_a^2 \\ 0 & 1 \end{pmatrix} \hookrightarrow \mathrm{GL}_{3,\IQ}.
\end{equation}

Writing $\mathrm{pr}_i: \mathcal{E}(\mathcal{N})^g \rightarrow \mathcal{E}(\mathcal{N})$ for the projection to the $i$-th factor, we endow $\mathcal{E}(\mathcal{N})^g(\IC)$ with the variation $H = \bigoplus_{j=1}^g\mathrm{pr}_j^\ast H^\prime$ of mixed $\IZ$-Hodge structures. The base change $H_\IQ$ can again be interpreted in Shimura-theoretic terms as arising from the $g$-fold product of the representation in \eqref{equation::product}. This allows to invoke a result of Wildeshaus \cite[Theorem II.2.2]{Wildeshaus1997} implying that $H$ is an admissible variation of mixed $\IZ$-Hodge structures. Finally, the restriction $H|_{X(\IC)}$ is the desired admissible variation of mixed $\IZ$-Hodge structures on $X(\IC)$. (By definition (see e.g.\ \cite[Definition 14.49]{Peters2008}), admissibility is trivially perserved by passing to analytic subvarieties.) In the following, we write $H|_X$ instead of $H|_{X(\IC)}$ to simplify our notation. We also write $H_x$ for the mixed Hodge structure associated with a point $x \in X(\IC)$.

\section{The generic Mumford-Tate group of $H|_X$}
\label{subsection::mumfordtategroup}

Write $\mathrm{MT}(H|_X)$ for the generic Mumford-Tate group of $H|_{X}$. It is clear that $\mathrm{MT}(H|_X) \subseteq (\mathbb{G}_{a,\IQ}^2 \rtimes \mathrm{GL}_{2,\IQ})^g$ -- both from the explicit description and the Shimura-theoretic formulation. There exists a countable union $\mathcal{Z} \subseteq X(\IC)$ of proper analytic subvarieties such that $\mathrm{MT}(H|_X)=\mathrm{MT}(H_x)$ for all $x \in X(\IC) \setminus \mathcal{Z}$ and that, for all points $x\in X(\IC)$, we have $\mathrm{MT}(H_x)\subseteq \mathrm{MT}(H|_X)$; we refer the reader to \cite[Section 4]{Andre1992} or \cite[Section 6]{Milne2013} for details. 

Analogous results are true for the generic Mumford-Tate group $\mathrm{MT}(H_{-1}|_X) \subseteq \GL_{2,\IQ}^g$ of the variation $H_{-1}|_X = W_{-1}/W_{-2}(H|_X)$ of pure $\IZ$-Hodge structures of weight $-1$. There is an evident surjective homomorphism $\mathrm{MT}(H|_X) \twoheadrightarrow \mathrm{MT}(H_{-1}|_X)$ between the generic Mumford-Tate groups; we hence determine $\mathrm{MT}(H_{-1}|_X)$ first. Its structure is mostly related to the presence or absence of generic isogenies between the factors of $A=\prod_{j=1}^g E_j$.

For this reason, we make a further assumption to simplify our notation: Write $\eta$ for the generic point of $S$. In the sequel, we may and do assume that there exist integers
\begin{equation*}
i_1 = 1 < i_2 < \cdots < i_{p+1} = g + 1
\end{equation*}
such that the elliptic curves
\begin{equation*}
E_{i_q, \eta}, E_{i_q+1, \eta}, \dots, E_{i_{q+1}-1, \eta}
\end{equation*}
are isogeneous for each $1\leq q \leq p$, and the elliptic curves
\begin{equation*}
E_{i_1,\eta}, E_{i_2,\eta}, \dots, E_{i_p,\eta}
\end{equation*}
are pairwise non-isogeneous. Set also $g_q = i_{q+1} - i_q$ for each $q \in \{1,\dots, p \}$. For the proof of Theorem \ref{theorem:bogomolov}, we can even assume that 
\begin{equation*}
E_{i_q, \eta} = E_{i_q+1, \eta} = \cdots = E_{i_{q+1}-1, \eta}
\end{equation*}
for each $1\leq q \leq p$. We also assume that there exists a $p^\prime \in \{0,\dots, p\}$ such that the families
\begin{equation*}
E_{i_1} \rightarrow S, \ E_{i_2} \rightarrow S, \ \dots, \ E_{i_{p^\prime}} \rightarrow S
\end{equation*}
are non-isotrivial, and the families
\begin{equation*}
E_{i_{p^\prime+1}} \rightarrow S, \ E_{i_{p^\prime+2}} \rightarrow S, \ \dots, \ E_{i_{p}} \rightarrow S
\end{equation*}
are constant. (In particular, all families are constant if $p^\prime=0$ and non-isotrivial if $p^\prime=p$.) We set $g^\prime = i_{p^\prime+1}-1$ and $A_{\mathrm{cst}} \times S = E_{g^\prime+1} \times_S \cdots \times_S E_{g}$. For any sufficiently generic point $s \in S(\IC)$ (i.e., $s$ is not contained in a countable union of proper analytic subvarieties of $S(\IC)$), the elliptic curves  
\begin{equation*}
E_{i_1,s}, E_{i_2,s}, \dots, E_{i_{p},s}
\end{equation*}
are pairwise non-isogeneous and none of the curves
\begin{equation*}
E_{i_1,s}, E_{i_2,s}, \dots, E_{i_{p^\prime},s}
\end{equation*}
has complex multiplication. Using \cite[Theorems B.53 and B.72]{Lewis1999}, we obtain for every point $x \in \pi^{-1}(s)$ that
\begin{equation}
\label{equation::genMT-1}
\mathrm{MT}(H_{-1,x}) = \Gm (\Delta_{g_1}(\mathrm{SL}_{2,\IQ}) \times \cdots \times \Delta_{g_{p^\prime}}(\mathrm{SL}_{2,\IQ}) \times \mathrm{Hg}(A_{\mathrm{cst}})) \subseteq \GL_{2,\IQ}^g
\end{equation}
where $\Delta_k: \SL_{2,\IQ} \rightarrow \SL_{2,\IQ}^k$, $k \in \IZ^{>0}$, denotes the diagonal map and $\mathrm{Hg}(A_{\mathrm{cst}})$ is the Hodge group of $A_{\mathrm{cst}}$, which we do not need to determine here. As $s$ is sufficiently generic, the subgroup in \eqref{equation::genMT-1} equals the generic Mumford-Tate group $\mathrm{MT}(H_{-1}|_X)$.

For each $x\in X$, we let $U_x$ denote the unipotent radical of $\mathrm{MT}(H_x)$. By \cite[Lemma 2.(c)]{Andre1992}, the sequence
\begin{equation*}
\begin{tikzcd} 
1 \ar[r] & U_x \ar[r] & \mathrm{MT}(H_{x}) \ar[r] & \mathrm{MT}(H_{-1,x}) \ar[r]  & 1
\end{tikzcd}
\end{equation*}
is exact. We claim that $\dim(U_x)=2g = 2 \sum_{q=1}^p g_q$ for a sufficiently general $x \in X(\IC)$. This leads immediately to
\begin{equation}
\label{equation::genMT-2}
\mathrm{MT}(H_x) = \Gm \prod_{q=1}^{p^\prime} \left( \mathbb{G}_{a,\IQ}^{2g_q} \rtimes \mathrm{SL}_{2,\IQ} \right) \times (\mathbb{G}_{a,\IQ}^{2(g-g^\prime)} \rtimes \mathrm{Hg}(A_{\mathrm{cst}})) \subseteq (\mathbb{G}_{a,\IQ}^2 \rtimes \mathrm{GL}_{2,\IQ})^g
\end{equation}
by comparing dimensions; here each copy of $\mathrm{SL_2}$ acts on the respective additive group $\mathbb{G}_{a,\IQ}^{2g_q} = \mathbb{G}_{a,\IQ}^2 \times \cdots \times \mathbb{G}_{a,\IQ}^2$ diagonally on each factor $\mathbb{G}_{a,\IQ}^2$. Similarly, $\mathrm{Hg}(A_{\mathrm{cst}})$ acts on $\mathbb{G}_{a,\IQ}^{2(g-g^\prime)}$ but we do not need to specify this action further. Again, it follows that the generic Mumford-Tate group $\mathrm{MT}(H|_X)$ is the group in \eqref{equation::genMT-2}.

The remaining claim follows by applying \cite[Proposition 1]{Andre1992} for the mixed Hodge structure $H_x$ at a sufficiently general point 
\begin{equation}
\label{equation::gen_point}
x = (x_1,\dots,x_g) \in X \subset E_1 \times_S \cdots \times_S E_g = A
\end{equation}
such that the elliptic curves $\mathcal{E}(\mathcal{N})_{\xi_{\mathcal{N}}(x_{i_q})}$, $1\leq q \leq p$, are pairwise non-isogeneous. We can freely assume that, as \eqref{equation::gen_point} varies, the generic rank of 
\begin{equation}
\label{equation::rank}
\mathrm{rank}_{R_q}(R_q x_{i_q} + \cdots + R_q x_{i_{q+1}-1}), \ R_q=\mathrm{End}(\mathcal{E}(\mathcal{N})_{\xi_{\mathcal{N}}(x_{i_q})}),
\end{equation}
equals $i_{q+1} - i_q = g_q$; for otherwise $X$ would be contained in a proper horizontal torsion coset of $A$, which contradicts the assumption made in the statement of (RBC). This allows us to choose \eqref{equation::gen_point} further such that, for all $q \in \{1,\dots, p \}$, the rank in \eqref{equation::rank} is $g_q$.

We invoke the said proposition from \cite{Andre1992} for the $1$-motive $[u: \IZ^g \rightarrow A_{\pi(x)}]$ where
\begin{equation*}
u: (n_1, n_2, \dots,n_g) \longmapsto (n_1 x_1 , n_2 x_2, \dots, n_g x_g).
\end{equation*}
The Zariski closure of $u(\IZ^g)$ is $A_{\pi(x)}$ by our assumptions \eqref{equation::rank}. Furthermore, we have 
\begin{equation*}
\mathrm{End}(A_{\pi(x)}) = 
R_1^{g_1 \times g_1} \times
\cdots \times
R_p^{g_p \times g_p}.
\end{equation*}
so that
\begin{multline}
\label{equation::unipotent_andre}
\mathrm{Hom}_{\mathrm{End}_\IQ(A_{\pi(x)})}
(\mathrm{End}_\IQ(A_{\pi(x)})\cdot u(\IZ^g),H_1(A_{\pi(x)},\IQ)) \\
=
\prod_{q=1}^p
\mathrm{Hom}_{R_{q,\IQ}}
(R_{q,\IQ} \cdot u_q(\IZ^{g_q}),H_1(A_{\pi(x)},\IQ))
\end{multline}
where $u_q: \IZ^{g_q} \rightarrow E_{i_q} \times E_{i_q+1} \times \cdots \times E_{i_{q+1}-1} = E_{i_q}^{g_q}$, $1 \leq q \leq p$, is defined by
\begin{equation*}
u_q(n_{i_q},n_{i_q+1},\dots,n_{i_{q+1}-1})=(n_{i_q} x_{i_q} , n_{i_q+1} x_{i_q+1}, \dots, n_{i_{q+1}-1} x_{i_{q+1}-1}).
\end{equation*}
If the elliptic curve $\mathcal{E}(\mathcal{N})_{\xi_{\mathcal{N}}(x_{i_q})}$, $q \in \{1,\dots, p\}$, has no complex multiplication, then \eqref{equation::rank} implies
\begin{align*}
\mathrm{Hom}_{R_{q,\IQ}}
(R_{q,\IQ} \cdot u_q(\IZ^{g_q}),H_1(A_{\pi(x)},\IQ)) 
&\approx \Hom_{\IQ^{g_q\times g_q}}(\IQ^{g_q\times g_q}\cdot u_q(\IZ^{g_q}), (\IQ^{2})^{g_p}) \\
&\approx \Hom_{\IQ}(\IQ \cdot u_q(\IZ^{g_q}),\IQ^2) \\
&\approx \Hom_{\IQ}(\IQ^{g_p},\IQ^2) \approx \IQ^{2g_p}.
\end{align*}
Similarly, if the elliptic curve $\mathcal{E}(\mathcal{N})_{\xi_{\mathcal{N}}(x_{i_q})}$, $q \in \{1,\dots, p\}$, has complex multiplication, then
\begin{align*}
\mathrm{Hom}_{R_{q,\IQ}}
(R_{q,\IQ} \cdot u_q(\IZ^{g_q}),H_1(A_{\pi(x)},\IQ)) 
&\approx \Hom_{R_{q,\IQ}^{g_q\times g_q}}(R_{q,\IQ}^{g_q\times g_q}\cdot u_q(\IZ^{g_q}),R_{q,\IQ}^{g_q}) \\
&\approx \Hom_{R_{q,\IQ}}(R_{q,\IQ} \cdot u_q(\IZ^{g_q}),R_{q,\IQ}) \\
&\approx \Hom_{R_{q,\IQ}}(R_{q,\IQ}^{g_p},R_{q,\IQ}) \approx R_{q,\IQ}^{g_q} \approx \IQ^{2g_q}
\end{align*}
by \eqref{equation::rank}. Thus the $\IQ$-dimension of \eqref{equation::unipotent_andre} is $2g$. By André's proposition, the dimension of $U_x$ equals the $\IQ$-dimension of the linear space \eqref{equation::unipotent_andre}, whence \eqref{equation::genMT-2}.
Finally, let us note that the generic derived Mumford-Tate group is
\begin{equation}
\label{equation::derivedMT}
\mathrm{MT}^{\mathrm{der}}(H|_X)=\prod_{q=1}^{p^\prime} \left( \mathbb{G}_{a,\IQ}^{2g_q} \rtimes \mathrm{SL}_{2,\IQ} \right) \times (\mathbb{G}_{a,\IQ}^{2(g-g^\prime)} \rtimes \mathrm{Hg}(A_{\mathrm{cst}}))^{\mathrm{der}}.
\end{equation}
In fact, it is a normal subgroup of $\mathrm{MT}(H|_X)$. Furthermore, its normal subgroup
\begin{equation*}
G = \mathrm{MT}^{\mathrm{der}}(H|_X) \cap \left(\prod_{q=1}^{p^\prime} \left( \mathbb{G}_{a,\IQ}^{2g_q} \rtimes \mathrm{SL}_{2,\IQ} \right) \times \{ e \}\right)
\end{equation*}
projects surjectively onto 
\begin{equation*}
\mathrm{MT}^{\mathrm{der}}(H_{-1,x}) = \Delta_{g_1}(\mathrm{SL}_{2,\IQ}) \times \cdots \times \Delta_{g_{p^\prime}}(\mathrm{SL}_{2,\IQ}) = \SL_{2,\IQ}^{p^\prime} \subseteq \SL_{2,\IQ}^{g^\prime} = (\GL_{2,\IQ}^{g^\prime})^{\mathrm{der}}.
\end{equation*}
The following lemma, which is also of use in the next section, yields \eqref{equation::derivedMT}.

\begin{lemma}
	\label{lemma::normality}
	Let $G \subseteq \prod_{q=1}^{p^\prime} ( \mathbb{G}_{a,\IQ}^{2g_q} \rtimes \mathrm{SL}_{2,\IQ})$ be a normal $\IQ$-algebraic subgroup projecting onto $\mathrm{SL}_{2,\IQ}^{p^\prime}$. Then, we have $G=\prod_{q=1}^{p^\prime} ( \mathbb{G}_{a,\IQ}^{2g_q} \rtimes \mathrm{SL}_{2,\IQ})$.
\end{lemma}

\begin{proof}
	We first consider the case $p^\prime=1$ and write $g$ instead of $g_1$. Note that for every $$(v_1,\dots,v_{g})\in (\IQbar^2)^{g}$$ and every $$(w_1,\dots,w_{g},\gamma) \in (\IQbar^2)^{g} \rtimes \SL_2(\IQbar),$$
	the conjugate
	\begin{equation*}
	\left(v_1,\dots,v_{g},\begin{pmatrix} 1 & 0 \\ 0 & 1 \end{pmatrix}\right)
	\cdot (w_1,\dots,w_{g},\gamma)
	\cdot \left(v_1,\dots,v_{g},\begin{pmatrix} 1 & 0 \\ 0 & 1 \end{pmatrix}\right)^{-1}
	\end{equation*}
	equals
	\begin{equation*}
	(v_1-\gamma(v_1) + w_1,\dots,v_{g}-\gamma(v_{g}) + w_{g},\gamma).
	\end{equation*}
	Choose now a $\gamma = \begin{pmatrix} a & b \\ c & d \end{pmatrix}\in \SL_2(\IQbar)$ such that 
	\begin{equation}
	\label{equation::invertible_matrix}
	\begin{pmatrix} 1 & 0 \\ 0 & 1\end{pmatrix} - \begin{pmatrix} a & b \\ c & d \end{pmatrix} = \begin{pmatrix} 1- a & -b \\ -c & 1-d \end{pmatrix}
	\end{equation}
	is invertible; an explicit admissible choice would be $a=d=0$ and $b = -c = 1$. By assumption, there exists some
	\begin{equation*}
	(w_1,\dots,w_g,\gamma) \in G(\IQbar).
	\end{equation*}
	Moreover, the normality of $G(\IQbar)$ implies that
	\begin{equation*}
	(v_1-\gamma(v_1) + w_1,\dots,v_{g}-\gamma(v_{g}) + w_{g},\gamma) \in G(\IQbar)
	\end{equation*}
	for all $(v_1,\dots,v_g) \in (\IQbar^2)^g$. The invertibility of \eqref{equation::invertible_matrix} implies that the preimage of $\gamma \in \SL_2(\IQbar)$ in $G(\IQbar) \subseteq \IQbar^{2g} \rtimes \mathrm{SL}_2(\IQbar)$ is of (algebraic) dimension $2g$ as each of the maps 
	\begin{equation*}
	\IQbar^2 \longrightarrow \IQbar^2: v_i \longmapsto v_i-\gamma(v_i), \ i \in \{1,\dots,g\},
	\end{equation*}
	is surjective. This means that the quotient map $q: \mathbb{G}_{a,\IQ}^{2g} \rtimes \mathrm{SL}_{2,\IQ} \rightarrow \mathrm{SL}_{2,\IQ}$ restricts to a surjective homomorphism $q|_G: G \rightarrow \mathrm{SL}_{2,\IQ}$ whose kernel is of dimension $2g$. This is only possible if $G= \mathbb{G}_{a,\IQ}^{2g} \rtimes \SL_{2,\IQ}$, whence the assertion of the lemma in case $p=1$.
	
	The general case $p>1$ can be proven similarly, working with a lifting of $(\gamma,\dots, \gamma) \in \SL_2(\IQbar^2)^p$ in $G(\IQbar)$ and using the above argument in parallel on each of the $p$ factors. This shows that the kernel of $q|_G: G \rightarrow \mathrm{SL}_{2,\IQ}^p$ is of dimension $\sum_{q=1}^p 2(i_{q+1} - i_q) = 2g$, which forces again the asserted equality.
\end{proof}

\section{The monodromy of $H|_X$}
\label{subsection::monodromy}

Let $x_0 \in X(\IC)$ be a sufficiently general point such that $\mathrm{MT}(H_{x_0})=\mathrm{MT}(H|_X)$ and let $\mathrm{Mon}(H|_X) \subseteq \mathrm{MT}(H|_X)(\IQ)$ denote the (ordinary) monodromy group at $x_0$ of the local system $\mathbb{V}_\IZ^g$ underlying $H|_X$. We write $\mathrm{Mon}^{\mathrm{alg}}(H|_X) \subseteq \mathrm{MT}(H|_X)$ for the connected component of its $\IQ$-algebraic closure (i.e., the (connected) algebraic monodromy group of $H|_X$ with base point $x_0$ in \cite{Andre1992}). By \cite[Theorem 1]{Andre1992}, the group $\mathrm{Mon}^{\mathrm{alg}}(H|_X)$ (resp.\ $\mathrm{Mon}^{\mathrm{alg}}(H_{-1}|_X)$) is a $\IQ$-normal subgroup of $\mathrm{MT}^{\mathrm{der}}(H|_X)$ (resp.\ $\mathrm{MT}^{\mathrm{der}}(H_{-1}|_X)$). (Note that the notion of ``good'' variation of mixed Hodge structures used in \cite{Andre1992} agrees with that of an admissible variation of mixed Hodge structures. In fact, the latter notion is even trivially stronger than the former, but a result of Kashiwara \cite[Theorem 4.5.2]{Kashiwara1986} allows also to prove the converse implication.)

In \cite{Habegger2012a}, it is proven that $\mathrm{Mon}^{\mathrm{alg}}(H_{-1}|_X) = \SL_{2,\IQ}^{p^\prime}$ in our situation (see the proof of Equation (10) in \textit{loc.cit.}). As the natural map $\mathrm{Mon}^{\mathrm{alg}}(H|_X) \rightarrow \mathrm{Mon}^{\mathrm{alg}}(H_{-1}|_X)$ is evidently surjective, we infer from another use of Lemma \ref{lemma::normality} that $\mathrm{Mon}^{\mathrm{alg}}(H|_X)$ projects onto $\prod_{q=1}^{p^\prime} ( \mathbb{G}_{a,\IQ}^{2g_q} \rtimes \mathrm{SL}_{2,\IQ})$. 

Consider again the connected component $\widetilde{X}$ of $\mathcal{u}^{-1}_{\mathrm{mixed}}(X)$
from Section \ref{subsection::coverings} and its stabilizer $\mathrm{Stab}(\widetilde{X})$ under the action of $(\mathbb{Z}^{2} \rtimes \Gamma(\mathcal{N}))^g$ on $(\mathbb{C} \times \mathcal{H}_1)^g$. Both $\mathrm{Stab}(\widetilde{X})$ and $\mathrm{Mon}(H|_X)$ are canonically subgroups of $(\IZ^2 \rtimes \Gamma(\mathcal{N}))^g$, and in fact they are equal to each other. This seems well-known, but we include the argument here for lack of reference and convenience of the reader. For this purpose, we choose an arbitrary lifting $\widetilde{x}_0 \in \widetilde{X}$ of the point $x_0 \in X(\IC)$. If $\gamma \in \Stab_\Gamma(\widetilde{X})$, then there exists a path $\widetilde{\phi}: [0, 1] \rightarrow \widetilde{X}$ with $\widetilde{\phi}(0) = \widetilde{x}_0$ and $\widetilde{\phi}(1) = \gamma \cdot \widetilde{x}_0$. Through transport along the $\IZ$-local system $\widetilde{\mathbb V}^g_{\IZ} = \mathcal{u}_{\mathrm{mixed}}^{\ast}(\mathbb{V}_{\IZ}^g)$, the path $\widetilde{\phi}$ induces a $\IZ$-linear map
\begin{equation}
\label{equation::parallel_transport}
\mathbb{V}_{\IZ,x_0}^g = \widetilde{\mathbb{V}}_{\IZ,\widetilde{x}_0}^g \longrightarrow \widetilde{\mathbb{V}}_{\IZ, \gamma \cdot \widetilde{x}_0}^g = \mathbb{V}_{\IZ,x_0}^g, 
\end{equation}
which can be seen to equal $\gamma \in \mathrm{MT}(H_x)$ by unraveling definitions. Thus $\gamma \in \mathrm{Mon}(H|_X)$. If conversely $g \in \mathrm{Mon}(H|_X)$ is induced by a path $\phi: [0,1] \rightarrow X$ with $\phi(0)=\phi(1)=x$, then the lifting $\widetilde{\gamma}: [0,1] \rightarrow \widetilde{X}$ with $\widetilde{\phi}(0)=\widetilde{x}_0$ yields a point $\widetilde{\phi}(1)=\gamma \cdot \widetilde{x}_0$ for some $\gamma \in (\mathbb{Z}^{2} \rtimes \Gamma(\mathcal{N}))^g$ . Considering again \eqref{equation::parallel_transport}, we infer $\gamma = g$. This means that the intersection $g \widetilde{X} \cap \widetilde{X}$ is non-empty. As both $\widetilde{X}$ and $g \widetilde{X}$ are connected components of $\mathcal{u}_{\mathrm{mixed}}^{-1}(X)$, we infer that actually $\widetilde{X} = g \widetilde{X}$, whence $g \in \Stab(\widetilde{X})$.

\section{Separating holomorphic and anti-holomorphic terms}
\label{subsection::separation}

The $(\IZ^2 \rtimes \Gamma(\mathcal{N}))$-invariance of the $(1,1)$-form in \eqref{equation::invariant_form} implies that, for every local chart $\chi: B_1(0)^d \rightarrow \widetilde{X}$ and every $\gamma \in \mathrm{Stab}(\widetilde{X})$, the equations \eqref{equation::differential_equation_1} associated with the charts $\chi: B_1(0)^d \rightarrow \widetilde{X}$ and $\gamma \circ \chi: B_1(0)^d \rightarrow \widetilde{X}$ are equivalent. In order to extract non-trivial information from monodromy, we pass to the product $\widetilde{X} \times \widetilde{X} \subset (\IC \times \mathcal{H}_1)^{2g}$ and exploit the fact that both holomorphic and anti-holomorphic terms appear in \eqref{equation::differential_equation_1}. The author owes this important idea to \cite[Subsection 5.2]{Andre2020}. 

We write $\pr_i: \widetilde{X} \times \widetilde{X} \rightarrow \widetilde{X}$, $i \in \{1, 2\}$, for the projection to the $i$-th factor. On $\widetilde{X} \times \widetilde{X}$, we consider the holomorphic functions
\begin{align*}
z_l^{\sharp} = z_l \circ \pr_1, \ \tau_l^{\sharp} = \tau_l \circ \pr_1, \ 1\leq l \leq g,
\end{align*}
and the antiholomorphic functions
\begin{align*}
\overline{z}_l^{\flat} = \overline{z}_l \circ \pr_2, \ \overline{\tau}_l^{\flat} = \overline{\tau}_l \circ \pr_2, \ 1\leq l \leq g.
\end{align*}
We next deduce from the equations \eqref{equation::differential_equation_1} on $\widetilde{X}$ new equations constraining the analytic subvariety $\widetilde{X} \times \widetilde{X}$. These equations actually change under the product action of $(\IZ^2 \rtimes \Gamma(\mathcal{N}))^{2g}$ on $(\IC \times \mathcal{H}_1)^{2g}$, so that we can use the action of $\Stab(\widetilde{X}) \times \Stab(\widetilde{X})$ to obtain non-trivial information. Let $\chi_1: B_1(0)^d \rightarrow \widetilde{X}$ (resp.\ $\chi_2: B_1(0)^d \rightarrow \widetilde{X}$) be a chart with local coordinates $w_1^\sharp,\dots, w_d^\sharp$ (resp.\ $w_1^\flat,\dots,w_d^\flat$) on $B_1(0)^d$. To increase readability, we write here $f$ (resp.\ $\partial f/ \partial w_m^\sharp$, $\partial f/ \partial \overline{w}_m^\flat$) instead of $f \circ (\chi_1,\chi_2)$ (resp.\ $\partial (f \circ (\chi_1,\chi_2))/\partial w_m^\sharp$, $\partial (f \circ (\chi_1,\chi_2))/\partial \overline{w}_m^\flat$) for functions $f$ on $\widetilde{X} \times \widetilde{X}$. For each $j \in \{1,\dots, g\}$, we set
\begin{equation*}
\mathbf B_j^\prime =  \left( (\tau_l^\sharp - \overline{\tau}_l^\flat) \cdot \frac{\del z_l^\sharp}{\del w_m^\sharp} - (z_l^\sharp - \overline{z}_l^\flat) \cdot \frac{\del 
	\tau_l^\sharp}{\del w_m^\sharp} \right)_{\substack{l\in \{1,\dots, g\}, l \neq j \\ m \in \{1,\dots, d\}}},
\end{equation*}
and
\begin{equation*}
\mathbf C_j^\prime =  \left( (\tau_l^\sharp - \overline{\tau}_l^\flat) \cdot \frac{\del \overline{z}_l^\flat}{\del \overline{w}_m^\flat} - (z_l^\sharp - \overline{z}_l^\flat) \cdot \frac{\del \overline{\tau}_l^\flat}{\del \overline{w}_m^\flat} \right)_{\substack{l\in \{1,\dots, g\}, l \neq j \\ m \in \{1,\dots, d\}}}.
\end{equation*}
(Note that $\del z_l^\sharp/\del w_m^\sharp$ and $\del \tau_l^\sharp/\del w_m^\sharp$ (resp.\ $\overline{\del z_l^\flat/\del w_m^\flat} = \del \overline{z}_l^\flat/\del \overline{w}_m^\flat$ and $\overline{\del \tau_l^\flat/\del w_m^\flat}=\del \overline{\tau}_l^\flat/\del \overline{w}_m^\flat$) are holomorphic (resp.\ antiholomorphic) functions on $B_1(0)^d \times B_1(0)^d$.)
We claim that, for every choice of charts $\chi_i: B_1(0)^d \rightarrow \widetilde{X}$ ($i \in \{1,2\}$), the product chart $$(\chi_1,\chi_2): B_1(0)^d \times B_1(0)^d \rightarrow \widetilde{X} \times \widetilde{X},$$ satisfies the relations
\begin{equation}
\label{equation::differential_equation_2}
r_j(\tau_j^\sharp-\overline{\tau}_j^\flat)^3 \det(\mathbf B_j^\prime) \det(\mathbf C_j^\prime) 
= 
r_k(\tau_k^\sharp-\overline{\tau}_k^\flat)^3 \det(\mathbf B_k^\prime) \det(\mathbf C_k^\prime), \ j,k \in \{1,\dots g\},
\end{equation}
on $B_1(0)^d \times B_1(0)^d$. With this aim in mind, we consider the set
\begin{equation*}
\mathcal{R} = \bigcup_{(\chi_1, \chi_2)} (\chi_1(0),\chi_2(0))
\end{equation*}
where $$(\chi_1: B_1(0)^d \rightarrow \widetilde{X},\chi_2: B_1(0)^d \rightarrow \widetilde{X})$$ ranges through all pairs of charts on $\widetilde{X}$ such that $\chi_1 \times \chi_2$ satisfies the relation \eqref{equation::differential_equation_2} at $(0,0)$. The following three statements are equivalent:
\begin{enumerate}
	\item $(\widetilde{x}_1,\widetilde{x}_2) \in \mathcal{R}$.
	\item There exists a pair of charts $(\chi_1: B_1(0)^d \rightarrow \widetilde{X},\chi_2: B_1(0)^d \rightarrow \widetilde{X})$ such that $(\widetilde{x}_1,\widetilde{x}_2) \in \im(\chi_1 \times \chi_2)$ and \eqref{equation::differential_equation_2} is satisfied at $(\chi_1,\chi_2)^{-1}(\widetilde{x}_1,\widetilde{x}_2)$.
	\item For every pair of charts $(\chi_1: B_1(0)^d \rightarrow \widetilde{X},\chi_2: B_1(0)^d \rightarrow \widetilde{X})$ such that $(\widetilde{x}_1,\widetilde{x}_2) \in \im(\chi_1 \times \chi_2)$, the equation \eqref{equation::differential_equation_2} is satisfied at $(\chi_1,\chi_2)^{-1}(\widetilde{x}_1,\widetilde{x}_2)$.
\end{enumerate}
Indeed, the implications $(3)\Rightarrow (2)$ and $(2) \Rightarrow (1)$ are trivial and the implication $(1)\Rightarrow (3)$ follows from the fact that \eqref{equation::differential_equation_2} is independent under transformations of the local coordinates $w_1^\sharp,\dots, w_d^\sharp,w_1^\flat,\dots,w_d^\flat$. From these equivalences, we deduce that $\mathcal{R}$ is locally cut out by a real-analytic equation. We also see that our above claim amounts to $\mathcal{R} = \widetilde{X} \times \widetilde{X}$.

Since $\widetilde{X} \times \widetilde{X}$ is irreducible (as a complex-analytic subset of $(\IC \times \mathcal{H}_1)^{2g}$), it suffices hence to show that $\mathcal{R}$ contains a non-empty open subset of $\widetilde{X} \times \widetilde{X}$. For this purpose, we consider an arbitrary point $(\widetilde{x},\widetilde{x}) \in \widetilde{X} \times \widetilde{X}$ on the diagonal. Let $\chi: B_1(0)^d \rightarrow \widetilde{X}$ be a chart such that $\chi(0)=\widetilde{x}$. By the above equivalences, the real-analytic set $(\chi, \chi)^{-1}(\mathcal{R}) \subseteq B_1(0)^d \times B_1(0)^d$ coincides with
\begin{multline*}
\mathcal{R}_{(\chi,\chi)} = \left\{ (w_1^\sharp,\dots,w_d^\sharp, w_1^\flat,\dots,w_d^\flat) \in B_1(0)^d \times B_1(0)^d \right. \\
 \left\vert \
\text{$(\chi \times \chi)$ satisfies \eqref{equation::differential_equation_2} at $(w_1^\sharp,\dots,w_d^\sharp, w_1^\flat,\dots,w_d^\flat)$}
\right\}.
\end{multline*}
Writing $\overline{(\cdot)}: B_1(0)^d \rightarrow B_1(0)^d$ for the component-wise complex conjugation, the set $(\id_{B_1(0)^d} \times \overline{(\cdot)})^{-1}(\mathcal{R}_{(\chi,\chi)})$ is a \textit{complex}-analytic subset of $B_1(0)^d \times B_1(0)^d$ as can be seen by inspecting \eqref{equation::differential_equation_2}; this relies on the following elementary fact: If $f(w_1,\dots,w_d) = \sum_{i_1,\dots,i_d} a_{i_1,\dots,i_d} w_1^{i_1}\cdots w_d^{i_d}$ is a holomorphic function on $B_1(0)^d$, then $\overline{f(\overline{w}_1,\dots,\overline{w}_d)}= \sum_{i_1,\dots,i_d} \overline{a_{i_1,\dots,i_d}} w_1^{i_1}\cdots w_d^{i_d}$ is holomorphic as well. Furthermore, the fact that the chart $\chi$ satisfies the previous equation \eqref{equation::differential_equation_1} implies immediately that
\begin{align*}
\Delta_{\mathrm{skew}} 
&= \{ (w_1,\dots, w_d, \overline{w}_1,\dots, \overline{w}_d) \ | \ (w_1,\dots, w_d) \in B_1(0)^d \} \\
&\subseteq (\id_{B_1(0)^d} \times \overline{(\cdot)})^{-1}(\mathcal{R}_{(\chi,\chi)});
\end{align*}
in fact, plugging in the local coordinates $(w_1,\dots,w_d,w_1,\dots,w_d)$ into the equation \eqref{equation::differential_equation_2} for the chart $(\chi \times \chi)$ yields the original equation \eqref{equation::differential_equation_1} for the chart $\chi$ back. As the smallest complex-analytic set of $B_1(0)^d \times B_1(0)^d$ containing $\Delta_{\mathrm{skew}}$ is $B_1(0)^d \times B_1(0)^d$, we infer that
\begin{equation*}
(\id_{B_1(0)^d} \times \overline{(\cdot)})^{-1}(\mathcal{R}_{(\chi,\chi)})=B_1(0)^d \times B_1(0)^d
\end{equation*}
and hence $\mathcal{R}_{(\chi,\chi)}=B_1(0)^d \times B_1(0)^d$, whence $\mathcal{R} = \widetilde{X} \times \widetilde{X}$. In other words, the equations \eqref{equation::differential_equation_2} are satisfied on all of $\widetilde{X} \times \widetilde{X}$.

\section{Enter monodromy}
\label{section::monodromy}

Let $\chi_i: B_1(0)^d \rightarrow \widetilde{X}$, $i \in \{1,2\}$, two local charts and $\gamma \in \mathrm{Stab}(\widetilde{X})$. As $\gamma(\widetilde{X}) = \widetilde{X}$, the composite $\gamma \circ \chi_1: B_1(0)^d \rightarrow \widetilde{X}$ is a local chart of $\widetilde{X}$ as well. Writing
\begin{equation*}
\gamma = \left(
\begin{pmatrix} m_1 \\ n_1 \end{pmatrix},
\begin{pmatrix} a_1 & b_1 \\ c_1 & d_1 \end{pmatrix},
\dots, 
\begin{pmatrix} m_g \\ n_g \end{pmatrix}, 
\begin{pmatrix} a_g & b_g \\ c_g & d_g \end{pmatrix} \right),
\end{equation*}
we note that
\begin{equation*}
z_l \circ \gamma = \frac{z_l+m_l+n_l\tau_l}{c_l\tau_l+d_l} \ \ \text{and} \ \ \tau_l \circ \gamma = \frac{a_l\tau_l+b_l}{c_l\tau_l+d_l}, \ l \in \{1, \dots, g \},
\end{equation*}
as functions on $\widetilde{X}$. We infer that
\begin{equation*}
z_l^\sharp \circ (\gamma,\id) = \frac{z_l^\sharp+m_l+n_l\tau_l^\sharp}{c_l\tau_l^\sharp+d_l}
\ \ \text{and} \ \
\tau_l^\sharp \circ (\gamma, \id) = \frac{a_l\tau_l^\sharp+b_l}{c_l\tau_l^\sharp+d_l}, \ l \in \{1, \dots, g \},
\end{equation*}
as functions on $\widetilde{X} \times \widetilde{X}$. We also need the derivatives
\begin{equation*}
\frac{\del (z_l^\sharp \circ (\gamma,\id))}{\del w_m^\sharp} = \frac{1}{(c_l\tau_l^\sharp+d_l)}\cdot \left(\frac{\del z_l^\sharp}{\del w_m^\sharp} + n_l \cdot \frac{\del \tau_l^\sharp}{\del w_m^\sharp}\right)
- \frac{c_l(z_l^\sharp+m_l+n_l\tau_l^\sharp)}{(c_l\tau_l^\sharp+d_l)^2} \cdot \frac{\del \tau_l^\sharp}{\del w_m^\sharp}
\end{equation*}
and
\begin{equation*}
\frac{\del (\tau_l^\sharp \circ (\gamma,\id))}{\del w_m^\sharp} = \frac{1}{(c_l\tau_l^\sharp+d_l)^2} \cdot \frac{\del \tau_l^\sharp}{\del w_m^\sharp}
\end{equation*}
on $B_1(0)^d \times B_1(0)^d$. With this preparation, we can compute the equations \eqref{equation::differential_equation_2} for the chart
\begin{equation*}
(\gamma \circ \chi_1, \chi_2): B_1(0)^d \times B_1(0)^d \longrightarrow \widetilde{X} \times \widetilde{X};
\end{equation*}
these are
\begin{multline}
\label{equation::differential_equation_3}
r_j\left(\frac{a_j\tau_j^\sharp+b_j}{c_j\tau_j^\sharp+d_j}-\overline{\tau}_j^\flat\right)^3 \det(\mathbf B_j^{\prime \prime}) \det(\mathbf C_j^{\prime \prime}) 
= 
r_k\left(\frac{a_k\tau_k^\sharp+b_k}{c_k\tau_k^\sharp+d_k}-\overline{\tau}_k^\flat\right)^3 \det(\mathbf B_k^{\prime \prime}) \det(\mathbf C_k^{\prime \prime}), \\ j,k \in \{1,\dots g\},
\end{multline}
with the $(d \times d)$-matrices 
\begin{equation*}
(\mathbf B_j^{\prime\prime})_{\substack{l\in \{1,\dots, g\}, l \neq j \\ m \in \{1,\dots, d\}}}
\text{\ \  and \ \ }
(\mathbf C_j^{\prime\prime})_{\substack{l\in \{1,\dots, g\}, l \neq j \\ m \in \{1,\dots, d\}}}
\end{equation*}
where
\begin{multline*}
(\mathbf B_j^{\prime\prime})_{lm} = 
\left(\frac{a_l\tau_l^\sharp+b_l}{c_l\tau_l^\sharp+d_l} - \overline{\tau}_l^\flat\right) \cdot \left( \frac{1}{(c_l\tau_l^\sharp+d_l)}\cdot \left[\frac{\del z_l^\sharp}{\del w_m^\sharp} + n_l \cdot \frac{\del \tau_l^\sharp}{\del w_m^\sharp}\right] - \frac{c_l(z_l^\sharp+m_l+n_l\tau_l^\sharp)}{(c_l\tau_l^\sharp+d_l)^2} \cdot \frac{\del \tau_l^\sharp}{\del w_m^\sharp} \right)  \\ - \left(\frac{z_l^\sharp+m_l+n_l\tau_l^\sharp}{c_l\tau_l^\sharp+d_l} - \overline{z}_l^\flat\right) \cdot \frac{1}{(c_l\tau_l^\sharp+d_l)^2} \cdot \frac{\del \tau_l^\sharp}{\del w_m^\sharp}
\end{multline*}
and
\begin{equation*}
(\mathbf C_j^{\prime\prime})_{lm} = \left(\frac{a_l\tau_l^\sharp+b_l}{c_l\tau_l^\sharp+d_l} - \overline{\tau}_l^\flat\right) \cdot \frac{\del \overline{z}_l^\flat}{\del \overline{w}_m^\flat} - \left(\frac{z_l^\sharp+m_l+n_l\tau_l^\sharp}{c_l\tau_l^\sharp+d_l} - \overline{z}_l^\flat\right) \cdot \frac{\del \overline{\tau}_l^\flat}{\del \overline{w}_m^\flat}.
\end{equation*}

In summary, each pair $(\chi_1,\chi_2)$ of charts $\chi_i: B_1(0)^d \rightarrow \widetilde{X}$, $i \in \{1,2\}$, does not only satisfy the equation \eqref{equation::differential_equation_2}, but also the equations \eqref{equation::differential_equation_3} for all $\gamma \in \mathrm{Stab}(\widetilde{X})$. Moreover, we can consider each of these equations at each point of $B_1(0)^d \times B_1(0)^d$ as an algebraic equation on $\gamma \in \mathrm{Stab}(\widetilde{X})$, giving a $\mathbb{C}$-algebraic hypersurface of $\mathrm{MT}(H|_X)$ containing the $\IQ$-rational points $\mathrm{Mon}(H|_X) = \mathrm{Stab}(\widetilde{X})$. By \cite[Corollary AG.14.6]{Borel1991}, this hypersurface contains also the algebraic monodromy group $\mathrm{Mon}^{\mathrm{alg}}(H|_X)$, which is the $\mathbb{Q}$-algebraic closure of these $\IQ$-rational points. We infer that each pair $(\chi_1,\chi_2)$ satisfies \eqref{equation::differential_equation_3} for all $\gamma \in \mathrm{Mon}^{\mathrm{alg}}(H|_X)(\IQ)$.

\section{A non-vanishing determinant}
\label{section::technical}

We make a final reduction before we start exploiting the equations \eqref{equation::differential_equation_3} obtained in the last section. To be precise, we show that we can assume the following: For each chart
\begin{equation*}
\chi_0: B_1(0)^d \longrightarrow \widetilde{X}, \
\underline{w} = (w_1,\dots,w_d) \longmapsto \left( z_l \circ \chi_0(\underline{w}), \tau_l \circ \chi_0(\underline{w}) \right)_{1\leq l \leq g},
\end{equation*}
the determinant
\begin{equation}
\label{equation::non_vanishing_determinant}
\det \left( \left( \frac{\del (z_l \circ \chi_0)}{\del w_m} \right)_{\substack{l\in \{1,\dots, g-1\} \\ m \in \{1,\dots, d\}}} \right)
\end{equation}
is a \textit{non-zero} holomorphic function on $B_1(0)^d$. Note that this condition holds for every chart if and only if it holds for a single one.

For each finite map $S^\prime \rightarrow S$, note that (RBC) for a subvariety $X$ in a family $\pi: A \rightarrow S$ is equivalent to (RBC) for the subvariety $X_{S^\prime} = X \times_S S^\prime$ in the family $\pi_{S^\prime}: A \times_S S^\prime \rightarrow S^\prime$. Furthermore, (RBC) for a subvariety $X \subseteq A$ is equivalent to (RBC) for any translate $X + \tau \subseteq A$ by a torsion section $\tau: S \rightarrow A$. By assumption, $S$ is a subvariety of $Y(\mathcal{N})$ and $A = \mathcal{E}(\mathcal{N})|_S$. 

Let $\underline{q}=(s_1/t_1,\dots,s_g/t_g) \in \IQ^g$ be given with $\gcd(s_i,t_i)=1$ for all $i \in \{1,\dots,g \}$, we set $\mathcal{N}^\prime = \mathrm{lcm}(t_1,\dots,t_g,\mathcal{N})$, so that all torsion points of order $\mathrm{lcm}(t_1,\dots,t_g)$ in the generic fiber $\mathcal{E}(\mathcal{N}^\prime)_{\eta_{Y(\mathcal{N}^\prime)}}$ extend to torsion sections $Y(\mathcal{N}^\prime) \rightarrow \mathcal{E}(\mathcal{N}^\prime)$. Writing $\xi_{\mathcal{N}^\prime/\mathcal{N}}: \mathcal{E}(\mathcal{N}^\prime) \rightarrow \mathcal{E}(\mathcal{N})$ for the standard covering and setting $X^\prime_{\underline{q}} = \xi_{\mathcal{N}^\prime/\mathcal{N}}^{-1}(X)$, the analytic variety $\widetilde{X}$ is also a connected component of $\mathcal{u}_{\mathrm{mixed}}^{-1}(X^\prime_{\underline{q}})$. Thus, there exists a torsion section $\tau: Y(\mathcal{N}^\prime) \rightarrow \mathcal{E}(\mathcal{N}^\prime)$ such that the translate
\begin{equation*}
\widetilde{X}_{\underline{q}} = \left\{ \left(z_i + \frac{s_i}{t_i} \cdot \tau_i,\tau_i\right)_{1\leq i \leq g} \in (\IC \times \mathcal{H}_1)^g \ \middle| \ \left(z_i,\tau_i\right)_{1\leq i \leq g} \in \widetilde{X} \right\}
\end{equation*}
is a connected component of $\mathcal{u}_{\mathrm{mixed}}^{-1}(X^\prime_{\underline{q}} + \tau)$.

For a fixed local chart
\begin{equation*}
\chi_0: B_1(0)^d \longrightarrow \widetilde{X}, \
\underline{w} \longmapsto \left( z_i \circ \chi_0(\underline{w}), \tau_i \circ \chi_0(\underline{w}) \right)_{1\leq i \leq g},
\end{equation*}
each of its translates 
\begin{equation*}
\chi_{0,\underline{q}}: B_1(0)^d \longrightarrow \widetilde{X}_{\underline{q}}, \
\underline{w} \longmapsto \left( z_i \circ \chi_0(\underline{w}) + \frac{s_i}{t_i} \cdot \tau_i, \tau_i \circ \chi_0(\underline{w}) \right)_{1\leq i \leq g},
\end{equation*}
is a chart of $\widetilde{X}_{\underline{q}}$. As (RBC) for $X^\prime_{\underline{q}}$ is equivalent to (RBC) for $X$ by the above remarks, it suffices to prove that the determinant \eqref{equation::non_vanishing_determinant} is a non-zero holomorphic function for a single chart $\chi_{0,\underline{q}}$ ($\underline{q} \in \IQ^g$). If this would not be the case, then
\begin{equation*}
\det \left( \left( \frac{\del (z_l \circ \chi_0)}{\del w_m} + \frac{s_l}{t_l} \cdot \frac{\del (\tau_l \circ \chi_0)}{\del w_m} \right)_{\substack{l\in \{1,\dots, g-1\} \\ m \in \{1,\dots, d\}}} \right) (\underline{w}) = 0
\end{equation*}
for all $(s_1/t_1,\cdots,s_g/t_g) \in \IQ^g$ and all $\underline{w} \in B_1(0)^d$. By continuity, this means
\begin{equation*}
\det \left( \left( \frac{\del (z_l \circ \chi_0)}{\del w_m} + u_l \cdot \frac{\del (\tau_l \circ \chi_0)}{\del w_m} \right)_{\substack{l\in \{1,\dots, g - 1\} \\ m \in \{1,\dots, d\}}} \right)(\underline{w}) = 0
\end{equation*}
for all $(u_1,\dots,u_g) \in \IR^g$ and all $\underline{w} \in B_1(0)^d$. In particular, we can take
\begin{equation*}
u_l = \frac{-\mathrm{Im}(z_l \circ \chi_0)}{\mathrm{Im}(\tau_l \circ \chi_0)}(\underline{w}), \ 1 \leq l \leq g,
\end{equation*}
so that
\begin{equation}
\label{equation::volume_zero}
\det \left( \left( \frac{\del (z_l \circ w_0)}{\del w_m} - \frac{\mathrm{Im}(z_l)}{\mathrm{Im}(\tau_l)} \cdot \frac{\del (\tau_l \circ \chi_0)}{\del w_m} \right)_{\substack{l\in \{1,\dots, g-1\} \\ m \in \{1,\dots, d\}}} \right)(\underline{w}) = 0
\end{equation}
for all $\underline{w} \in B_1(0)^d$. Comparing with \eqref{equation::alpha_j_prime}, we infer that $\alpha_{1}^\prime = 0$, which is a clear contradiction to the non-degeneracy of $\pr_{\widehat{1}}(X)$.

In summary, we can use without loss of generality that \eqref{equation::non_vanishing_determinant} is a non-zero holomorphic function for each chart $\chi: B_1(0)^d \rightarrow \widetilde{X}$ and a fixed $j_0 \in \{1, \dots, g\}$.

\section{Proof that $p^\prime = 0$ or $p^\prime = p$}
\label{section::reduction1}

We claim that either all the families $\pi: E_i \rightarrow S$ are constant (i.e., $p^\prime=0$) or non-constant (i.e., $p^\prime=p$). For this purpose, assume that $p^\prime\geq 1$ so that the family $E_{1} = \cdots = E_{i_2-1} \rightarrow S$ is non-isotrivial and that the family $E_{i_p} = \cdots = E_{g} \rightarrow S$ is constant. 

From Section \ref{subsection::monodromy}, we know that $\mathrm{Mon}^{\mathrm{alg}}(H|_X) \subseteq (\mathbb{G}_{a,\IQ}^2 \rtimes \mathrm{SL}_{2,\IQ})^g$ projects onto the first $p^\prime$ factors $\prod_{q=1}^{p^\prime} ( \mathbb{G}_{a,\IQ}^{2g_q} \rtimes \mathrm{SL}_{2,\IQ})$, which gives rise to a surjective map between their $\IQbar$-points. For every integer $N$, there hence exists an element
\begin{equation*}
\gamma =
\left(
\begin{pmatrix} m_1 \\ n_1 \end{pmatrix},
\begin{pmatrix} a_1 & b_1 \\ c_1 & d_1 \end{pmatrix},
\dots, 
\begin{pmatrix} m_g \\ n_g \end{pmatrix}, 
\begin{pmatrix} a_g & b_g \\ c_g & d_g \end{pmatrix} \right) 
\in \mathrm{Mon}^{\mathrm{alg}}(H|_X)(\IQbar)
\end{equation*} 
with
\begin{equation*}
a_l = 1, \ b_l = N, \ c_l = 0, \ d_l = 1, \ m_l = 0, \ n_l = 0
\end{equation*}
for all $1 \leq l \leq i_{p^\prime+1}-1$. Since the families $E_l$ ($i_{p^\prime+1} \leq l \leq g$) are trivial, we have furthermore
\begin{equation*}
a_l = 1, \ b_l = 0, \ c_l = 0, \ d_l = 1
\end{equation*}
for all $i_{p^\prime+1} \leq l \leq g$. Let us additionally fix a chart $\chi_0: B_1(0)^d \rightarrow \widetilde{X}$. Specializing to the chart $(\gamma \circ \chi_0, \chi_0): B_1(0)^d \times B_1(0)^d \rightarrow \widetilde{X} \times \widetilde{X}$, equation \eqref{equation::differential_equation_3} for $j=1$ and $k=g$ becomes
\begin{equation}
\label{equation::differential_equation_4}
r_1(\tau_1^\sharp+N-\overline{\tau}_1^\flat)^3 \det(\mathbf B_1^{\prime \prime}) \det(\mathbf C_1^{\prime \prime}) 
= 
r_g(\tau_g^\sharp-\overline{\tau}_g^\flat)^3 \det(\mathbf B_g^{\prime \prime}) \det(\mathbf C_g^{\prime \prime})
\end{equation}
with
\begin{equation*}
\mathbf B_j^{\prime\prime} = \left( (\tau_l^\sharp + b_l - \overline{\tau}_l^\flat) \cdot \frac{\del z_l^\sharp}{\del w_m^\sharp} - (z_l^\sharp - \overline{z}_l^\flat) \cdot \frac{\del \tau_l^\sharp}{\del w_m^\sharp}\right)_{\substack{l\in \{1,\dots, g\}, l \neq j \\ m \in \{1,\dots, d\}}}
\end{equation*}
and
\begin{equation*}
\mathbf C_j^{\prime\prime} = \left((\tau_l^\sharp+b_l - \overline{\tau}_l^\flat) \cdot \frac{\del \overline{z}_l^\flat}{\del \overline{w}_m^\flat} - (z_l^\sharp - \overline{z}_l^\flat) \cdot \frac{\del \overline{\tau}_l^\flat}{\del \overline{w}_m^\flat}\right)_{\substack{l\in \{1,\dots, g\}, l \neq j \\ m \in \{1,\dots, d\}}}.
\end{equation*}
We can consider the difference between the left-hand and the right-hand side of \eqref{equation::differential_equation_4} as a polynomial over the ring of (real-analytic) functions on $B_1(0)^d \times B_1(0)^d$ and indeterminate $N$. As this polynomial vanishes for each integer, it has to vanish identically. Expanding the left-hand side of the equation for the term of highest order in $N$, we note that its term of highest degree in $N$ is
\begin{equation*}
r_1 \prod_{l = i_{p^\prime+1}}^g (\tau_l^\sharp - \overline{\tau}^\flat_l)^2 \cdot \det \left( \left( \frac{\del z_l^\sharp}{\del w_m^\sharp}\right)_{\substack{l\in \{2,\dots, g\} \\ m \in \{1,\dots, d\}}} \right) \det \left( \left( \frac{\del \overline{z}_l^\flat}{\del \overline{w}_m^\flat}\right)_{\substack{l\in \{2,\dots, g\} \\ m \in \{1,\dots, d\}}} \right) N^{2i_{p^\prime +1}-1}
\end{equation*}
by our assumption on the non-vanishing of \eqref{equation::non_vanishing_determinant}.
However, the leading term on the right-hand side has degree $\leq 2(i_{p^\prime+1}-1) = 2i_{p^\prime+1}-2$. From this contradiction, we conclude that $p^\prime=0$ or $p^\prime = p$. Note that by our assumptions in Section \ref{subsection::reductions}, $p^\prime=0$ implies $\dim(S)=0$. This means that $A$ is just an abelian variety, for which (RBC) is proven in \cite{Zhang1998}. We hence concentrate on the case $p^\prime = p$ in the following.

\section{Existence of generic isogenies}
\label{subsection::diagonal}

In this section, we prove that the generic fibers $E_{i,\eta}$ ($1\leq i \leq g$) are all isogeneous (i.e., $p=1$) in the remaining case that all families $E_i \rightarrow S$ ($1\leq i \leq g$) are non-isotrivial (i.e., $p^\prime = p$). Let $M_q$, $1\leq q \leq p$, be arbitrary integers and set
\begin{equation*}
N_{i_q} = N_{i_q+1} = \cdots = N_{i_{q+1}-1} = M_q
\end{equation*}
for all $1 \leq q \leq p$. Again by the results from Section \ref{subsection::monodromy}, we know that $\mathrm{Mon}^{\mathrm{alg}}(H|_X) = \prod_{q=1}^{p} ( \mathbb{G}_{a,\IQ}^{2g_q} \rtimes \mathrm{SL}_{2,\IQ})$ (diagonally embedded in $(\mathbb{G}_{a,\IQ}^2 \rtimes \mathrm{SL}_{2,\IQ})^g$). Thus, there exists an element
\begin{equation*}
\gamma =
\left(
\begin{pmatrix} m_1 \\ n_1 \end{pmatrix},
\begin{pmatrix} a_1 & b_1 \\ c_1 & d_1 \end{pmatrix},
\dots, 
\begin{pmatrix} m_g \\ n_g \end{pmatrix}, 
\begin{pmatrix} a_g & b_g \\ c_g & d_g \end{pmatrix} \right) 
\in \mathrm{Mon}^{\mathrm{alg}}(H|_X)(\IQ)
\end{equation*} 
with
\begin{equation*}
a_l = 1, \ b_l = N_l, \ c_l = 0, \ d_l = 1, \ m_l = 0, \ n_l = 0
\end{equation*}
for all $1 \leq l \leq g$. Since there is nothing to prove if $p=1$, we can assume that there exists $k \in \{i_2,\dots, i_3 -1 \}$. Specializing again to the chart $(\gamma \circ \chi_0, \chi_0): B_1(0)^d \times B_1(0)^d \rightarrow \widetilde{X} \times \widetilde{X}$, the equation \eqref{equation::differential_equation_3} for $j=1$ and $k$ as here becomes
\begin{multline}
\label{equation::differential_equation_5}
r_1(\tau_1^\sharp+N_1-\overline{\tau}_1^\flat)^3 \det(\mathbf B_1^{\prime \prime}) \det(\mathbf C_1^{\prime \prime}) 
= 
r_k(\tau_k^\sharp+N_k-\overline{\tau}_k^\flat)^3 \det(\mathbf B_k^{\prime \prime}) \det(\mathbf C_k^{\prime \prime})
\end{multline}
with the $(d \times d)$-matrices
\begin{equation*}
\mathbf B_j^{\prime\prime} = \left( (\tau_l^\sharp+N_l - \overline{\tau}_l^\flat) \cdot \frac{\del z_l^\sharp}{\del w_m^\sharp} - (z_l^\sharp - \overline{z}_l^\flat) \cdot \frac{\del \tau_l^\sharp}{\del w_m^\sharp} \right)_{\substack{l\in \{1,\dots, g\}, l \neq j \\ m \in \{1,\dots, d\}}}
\end{equation*}
and
\begin{equation*}
\mathbf C_j^{\prime\prime} = \left( (\tau_l^\sharp+N_l - \overline{\tau}_l^\flat) \cdot \frac{\del \overline{z}_l^\flat}{\del \overline{w}_m^\flat} - (z_l^\sharp - \overline{z}_l^\flat)  \cdot \frac{\del \overline{\tau}_l^\flat}{\del \overline{w}_m^\flat} \right)_{\substack{l\in \{1,\dots, g\}, l \neq j \\ m \in \{1,\dots, d\}}}.
\end{equation*}
The difference of both sides in \eqref{equation::differential_equation_5} can be considered as a multivariate polynomial in the variables $M_q = N_{i_q} = \cdots = N_{i_{q+1}-1} $, $1\leq q \leq p$, of total degree $\leq (2d+3)$, which has to vanish identically because its evaluations for  all $(M_1,\dots,M_p) \in \IZ^p$ vanish. In fact, the sum of its terms of highest degree $2d+3$ is
\begin{multline*}
r_1 \det \left( \left( \frac{\del z_l^\sharp}{\del w_m^\sharp}\right)_{\substack{l\in \{2,\dots, g\} \\ m \in \{1,\dots, d\}}} \right) \det \left( \left( \frac{\del \overline{z}_l^\flat}{\del \overline{w}_m^\flat}\right)_{\substack{l\in \{2,\dots, g\} \\ m \in \{1,\dots, d\}}} \right) \cdot M_1 \prod_{l=1}^g N_l^2 \\
-
r_k \det \left( \left( \frac{\del z_l^\sharp}{\del w_m^\sharp}\right)_{\substack{l\in \{1,\dots, g\}, l \neq k \\ m \in \{1,\dots, d\}}} \right) \det \left( \left( \frac{\del \overline{z}_l^\flat}{\del \overline{w}_m^\flat}\right)_{\substack{l\in \{1,\dots, g\}, l \neq k \\ m \in \{1,\dots, d\}}} \right) \cdot M_2 \prod_{l=1}^g N_l^2.
\end{multline*}
Since the two terms contain different monomials in $M_1,\dots,M_p$, their coefficients must vanish. As in the previous section, the vanishing of the first coefficient yields a contradiction to the non-vanishing of \eqref{equation::non_vanishing_determinant}, whence $p=1$. Recall that this implies
\begin{equation*}
E_1 = E_2 = \cdots = E_g
\end{equation*}
by the assumptions made in Subsection \ref{subsection::mumfordtategroup}. Our assumptions on $S$ from Section \ref{subsection::reductions} imply furthermore that $S$ is the diagonal of $\mathcal{E}(\mathcal{N})^g$ in this case.

\section{Existence of a linear equation on $z_l|_{\widetilde{X}}$ ($1\leq l \leq g$).}
\label{section::linear}

In this section, we deduce a linear equation governing the restrictions of the functions $z_l$, $1\leq l \leq g$, to $\widetilde{X}$. For each integer $N$, there exists an element
\begin{equation*}
\gamma =
\left(
\begin{pmatrix} m_1 \\ n_1 \end{pmatrix},
\begin{pmatrix} a_1 & b_1 \\ c_1 & d_1 \end{pmatrix},
\dots, 
\begin{pmatrix} m_g \\ n_g \end{pmatrix}, 
\begin{pmatrix} a_g & b_g \\ c_g & d_g \end{pmatrix} \right) 
\in \mathrm{Mon}^{\mathrm{alg}}(H|_X)(\IQ)
\end{equation*} 
with
\begin{equation*}
a_l = 1, \ b_l = N, \ c_l = 0, \ d_l = 1, \ m_l = 0, \ n_l = 0
\end{equation*}
for all $1 \leq l \leq g$. Specializing again to the chart $$(\gamma \circ \chi_0, \chi_0): B_1(0)^d \times B_1(0)^d \rightarrow \widetilde{X} \times \widetilde{X},$$ the equations \eqref{equation::differential_equation_3} become
\begin{multline}
\label{equation::differential_equation_6}
r_j(\tau_j^\sharp+N-\overline{\tau}_j^\flat)^3 \det(\mathbf B_j^{\prime \prime}) \det(\mathbf C_j^{\prime \prime}) 
= 
r_k(\tau_k^\sharp+N-\overline{\tau}_k^\flat)^3 \det(\mathbf B_k^{\prime \prime}) \det(\mathbf C_k^{\prime \prime}), \\ j,k \in \{1,\dots, g\},
\end{multline}
with the $(d \times d)$-matrices
\begin{equation*}
\mathbf B_j^{\prime\prime} = \left( (\tau_l^\sharp+N - \overline{\tau}_l^\flat) \cdot \frac{\del z_l^\sharp}{\del w_m^\sharp} - (z_l^\sharp - \overline{z}_l^\flat) \cdot \frac{\del \tau_l^\sharp}{\del w_m^\sharp} \right)_{\substack{l\in \{1,\dots, g\}, l \neq j \\ m \in \{1,\dots, d\}}}
\end{equation*}
and
\begin{equation*}
\mathbf C_j^{\prime\prime} = \left( (\tau_l^\sharp+N - \overline{\tau}_l^\flat) \cdot \frac{\del \overline{z}_l^\flat}{\del \overline{w}_m^\flat} - (z_l^\sharp - \overline{z}_l^\flat)  \cdot \frac{\del \overline{\tau}_l^\flat}{\del \overline{w}_m^\flat} \right)_{\substack{l\in \{1,\dots, g\}, l \neq j \\ m \in \{1,\dots, d\}}}.
\end{equation*}
Regarding again the difference of both sides of \eqref{equation::differential_equation_6} as a polynomial in the indeterminate $N$ with coefficients in the ring of (real-analytic) functions on $B_1(0)^d \times B_1(0)^d$ and considering the term of highest degree, we obtain
\begin{multline*}
r_j \det \left( \left( \frac{\del z_l^\sharp}{\del w_m^\sharp}\right)_{\substack{l\in \{1,\dots, g\}, l \neq j \\ m \in \{1,\dots, d\}}} \right) \det \left( \left( \frac{\del \overline{z}_l^\flat}{\del \overline{w}_m^\flat}\right)_{\substack{l\in \{1,\dots, g\}, l \neq j \\ m \in \{1,\dots, d\}}} \right) \\
=
r_k \det \left( \left( \frac{\del z_l^\sharp}{\del w_m^\sharp}\right)_{\substack{l\in \{1,\dots, g\}, l \neq k \\ m \in \{1,\dots, d\}}} \right) \det \left( \left( \frac{\del \overline{z}_l^\flat}{\del \overline{w}_m^\flat}\right)_{\substack{l\in \{1,\dots, g\}, l \neq k \\ m \in \{1,\dots, d\}}} \right)
\end{multline*}
Specializing to the diagonal $B_1(0)^d \subset B_1(0)^d \times B_1(0)^d$, this yields
\begin{equation*}
r_j
\left\vert \det \left( \begin{pmatrix}
\frac{\del z_l}{\del w_m}
\end{pmatrix}_{\substack{l\in \{1,\dots, g\}, l \neq j \\ m \in \{1,\dots, d\}}} \right)
\right\vert^2
= r_k
\left\vert \det \left(\begin{pmatrix}
\frac{\del z_l}{\del w_m}
\end{pmatrix}_{\substack{l\in \{1,\dots, g\}, l \neq k \\ m \in \{1,\dots, d\}}} \right)
\right\vert^2
\end{equation*}
for the chart $\chi_0: B_1(0)^d \rightarrow \widetilde{X}$. As both determinants are holomorphic functions on $B_1(0)^d$, this implies
\begin{equation}
\label{equation::delz}
r_j^{1/2} \det \left( \begin{pmatrix}
\frac{\del z_l}{\del w_m}
\end{pmatrix}_{\substack{l\in \{1,\dots, g\}, l \neq j \\ m \in \{1,\dots, d\}}} \right) = u_{j,k} \cdot r_k^{1/2}  \det \left(\begin{pmatrix}
\frac{\del z_l}{\del w_m}
\end{pmatrix}_{\substack{l\in \{1,\dots, g\}, l \neq k \\ m \in \{1,\dots, d\}}} \right)
\end{equation}
for some $u_{j,k} \in S^1 = \{ z \in \IC \ | \ |z|=1\}$. Since the determinant \eqref{equation::non_vanishing_determinant} does not vanish, we can specialize \eqref{equation::delz} to $k=g$ and obtain
\begin{equation*}
\frac{\det \begin{pmatrix}
	\frac{\del z_l}{\del w_m}
	\end{pmatrix}_{\substack{l\in \{1,\dots, g\}, l \neq j \\ m \in \{1,\dots, d\}}}}
	{\det \begin{pmatrix}
	\frac{\del z_l}{\del w_m}
	\end{pmatrix}_{\substack{l\in \{1,\dots, g-1\} \\ m \in \{1,\dots, d\}}}} = \frac{u_{j,g} \cdot r_g^{1/2}}{r_j^{1/2}}
\end{equation*}
for each $j \in \{1,\dots, g-1 \}$. Setting $f_j= (-1)^{(g-l-1)}u_{j,g}r_g^{1/2}/r_j^{1/2} \in \IC^\times$ for $j \in \{1,\dots,g-1\}$, we can use Kramer's rule to obtain
\begin{equation*}
f_1
\begin{pmatrix}
\frac{\del z_1}{\del w_1} \\
\frac{\del z_1}{\del w_2} \\
\cdots \\
\frac{\del z_1}{\del w_d} 
\end{pmatrix}
+
f_2
\begin{pmatrix}
\frac{\del z_2}{\del w_1} \\
\frac{\del z_2}{\del w_2} \\
\cdots \\
\frac{\del z_2}{\del w_d} 
\end{pmatrix}
+
\cdots
+
f_{g-1}
\begin{pmatrix}
\frac{\del z_{g-1}}{\del w_1} \\
\frac{\del z_{g-1}}{\del w_2} \\
\cdots \\
\frac{\del z_{g-1}}{\del w_d} 
\end{pmatrix}
=
\begin{pmatrix}
\frac{\del z_{g}}{\del w_1} \\
\frac{\del z_{g}}{\del w_2} \\
\cdots \\
\frac{\del z_{g}}{\del w_d} 
\end{pmatrix}
\end{equation*}
on $B_1(0)^d$. Therefore, we obtain
\begin{equation*}
\frac{\del}{\del w_m} \left( f_1 z_1 + f_2 z_2 + \dots + f_{g-1} z_{g-1} + z_g \right) = 0, \ 1 \leq m \leq d,
\end{equation*}
for all $m \in \{1, \dots, d\}$. In conclusion, we obtain a non-trivial linear equation
\begin{equation}
\label{equation::linear_equation}
f_1 z_1 + f_2 z_2 + \dots + f_g z_g = b, \ b \in \IC,
\end{equation}
valid on all of $\widetilde{X}$ by real-analytic continuation.

\section{Completion of the proof of Theorem \ref{theorem:bogomolov}}
\label{section::last}

By Section \ref{subsection::diagonal}, the variety $S$ is the diagonal in $Y(\mathcal{N})^g$ and hence a special Shimura subvariety. Thus by \cite[Section 3.3]{Gao2018}, the bi-algebraic closure $X^{\mathrm{biZar}}$ of $X$ as defined in \cite{Gao2018} is the minimal horizontal torsion coset containing $X$. It is therefore our goal to prove that $\dim(X^{\mathrm{biZar}}) \leq g$ by means of the Ax-Schanuel conjecture for mixed Shimura varieties; in fact, this contradicts the assumption in (RBC). For this purpose, we set
\begin{equation*}
Y = \{ (\widetilde{x},x) \in \widetilde{X} \times X(\IC) \ | \ \mathcal{u}_{\mathrm{mixed}}(\widetilde{x})=x\} \subseteq (\mathcal{H}_1 \times \IC)^g \times \mathcal{E}(\mathcal{N})^g(\IC).
\end{equation*}
The mixed Ax-Schanuel conjecture in the form of \cite[Theorem 1.1]{Gao2018} yields
\begin{equation*}
\dim(X^{\mathrm{biZar}}) \leq \dim(Y^{\mathrm{Zar}}) - \dim(Y)
\end{equation*}
where $Y^{\mathrm{Zar}}$ is the Zariski closure of $Y$ in $(\IP^1(\IC) \times \IC)^g \times \mathcal{E}(\mathcal{N})^g(\IC)$. Writing $H \subset \IC^g$ for the linear hypersurface determined by \eqref{equation::linear_equation} and $\Delta(\mathbb{P}^1(\IC))$ for the diagonal in $\mathbb{P}^1(\IC)$, the analytic variety $Y$ is contained in the algebraic subset 
\begin{equation*}
(\Delta(\mathcal{H}_1) \times H) \times X \subset (\IP^1(\IC) \times \IC)^g \times \mathcal{E}(\mathcal{N})^g(\IC).
\end{equation*}
It follows that
\begin{align*}
\dim(Y^{\mathrm{Zar}}) \leq 1 + (g-1) + \dim(X) = \dim(X) + g.
\end{align*}
As $\dim(Y) = \dim(X)$, we conclude that
\begin{equation*}
\dim(X^{\mathrm{biZar}}) \leq \dim(Y^{Zar}) - \dim(Y) \leq g,
\end{equation*}
which concludes our proof of Theorem \ref{theorem:bogomolov}.

\textbf{Acknowledgements:} The author thanks Laura DeMarco, Ziyang Gao, Philipp Habegger, Myrto Mavraki, and Fabien Pazuki for their advice, discussion and encouragement. Finally, he thanks the anonymous referee for their attentive reading and their many suggestions that helped to improve the exposition substantially. He also thanks Jakob Stix for pointing out some inaccuracies in the article.

\bibliographystyle{plain}
\bibliography{references}

\end{document}